\documentclass[a4paper,12pt]{amsart}

\usepackage{amsmath}
\usepackage{amssymb}
\usepackage{amsthm}
\usepackage[mathscr]{eucal}

\newcommand{\ANR}{\operatorname{ANR}}
\newcommand{\Ant}{\operatorname{Ant}}
\newcommand{\CAT}{\operatorname{CAT}}
\newcommand{\rad}{\operatorname{rad}}
\newcommand{\GCBA}{\operatorname{GCBA}}
\newcommand{\R}{\operatorname{\mathbb{R}}}
\newcommand{\Z}{\operatorname{\mathbb{Z}}}
\newcommand{\N}{\operatorname{\mathbb{N}}}
\newcommand{\Sph}{\operatorname{\mathbb{S}}}
\newcommand{\GH}{\operatorname{\mathrm{GH}}}
\newcommand{\T}{\operatorname{\mathrm{T}}}
\newcommand{\G}{\operatorname{\mathcal{G}}}
\newcommand{\Haus}{\operatorname{\mathcal{H}}}
\numberwithin{equation}{section}

\theoremstyle{plain}
\newtheorem{thm}{Theorem}[section]
\newtheorem{lem}[thm]{Lemma}  
\newtheorem{prop}[thm]{Proposition}

\theoremstyle{definition}
\newtheorem{defn}{Definition}[section]
\newtheorem{exmp}{Example}[section]
\newtheorem{prob}{Problem}[section]

\theoremstyle{remark}
\newtheorem{rem}{Remark}[section]

\begin{document} 

\title
[Asymptotic topological regularity of CAT(0) spaces]
{Asymptotic topological regularity of CAT(0) spaces}

\author
[Koichi Nagano]{Koichi Nagano}

\thanks{
Partially supported
by JSPS KAKENHI Grant Numbers 18740023, 20K03603}

\address
[Koichi Nagano]
{\endgraf 
Institute of Mathematics, University of Tsukuba
\endgraf
Tennodai 1-1-1, Tsukuba, Ibaraki, 305-8571, Japan}

\email{nagano@math.tsukuba.ac.jp}

\date{December 30, 2021}


\begin{abstract}
We study asymptotic topological regularity of $\CAT(0)$ spaces.
We prove that if a purely $n$-dimensional, proper, geodesically complete 
$\CAT(0)$ space has small volume growth,
then it is homeomorphic to the $n$-dimensional Euclidean space.
We also discuss asymptotic geometry of
proper, geodesically complete 
$\CAT(0)$ spaces of small volume growth.
\end{abstract}

\maketitle


\section{Introduction}

A $\CAT(0)$ space is defined as a geodesic metric space
that is globally non-positively curved in the sense of Alexandrov.
Since Gromov \cite{gromovh} formulated the definition,
$\CAT(0)$ spaces have played central roles
in geometry of metric spaces of non-positive curvature 
(see e.g., \cite{bridson-haefliger}).

A Hadamard manifold is by definition a simply connected,
complete Riemannian manifold of non-positive sectional curvature.
The Cartan--Hadamard theorem in Riemannian geometry tells us that
every $n$-dimensional Hadamard manifold is diffeomorphic 
to the $n$-dimensional Euclidean space $\R^n$.
A connected, complete Riemannian manifold
is $\CAT(0)$ if and only if it is a Hadamard manifold.

Gromov \cite{gromovq} asked the question whether 
there exists a convex geodesic metric space 
that is a topological $n$-manifold differ from $\R^n$.
We notice that every convex geodesic metric space is contractible.
Every $\CAT(0)$ space is convex.
For the case of $n \ge 5$,
Davis--Januszkiewicz \cite[Theorem 5b.1]{davis-januszkiewicz}
gave an affirmative answer,
in fact,
showed that for each $n \in \N$ with $n \ge 5$,
there exists a piecewise flat $\CAT(0)$ polyhedron
that is a topological $n$-manifold not homeomorphic to $\R^n$
(see also \cite{ancel-davis-guilbault}).
In the case of $n \le 3$,
if a convex geodesic metric space is a topological $n$-manifold,
then it is homeomorphic to $\R^n$ (\cite{brown}, \cite{rolfsen}).
Thurston \cite[Theorem 1.6]{thurston} proved that
if a $\CAT(0)$ topological $4$-manifold possesses a tame point,
then it is homeomorphic to $\R^4$.
As far as the author knows,
the question of Gromov \cite{gromovq} for $\CAT(0)$ spaces
remained open for the case of $n = 4$.
(see e.g., \cite[Section 5]{davis-januszkiewicz-lafont}).
Lytchak, Stadler, and the author \cite{lytchak-nagano-stadler}
recently prove that
every $\CAT(0)$ topological $4$-manifold
is homeomorphic to $\R^4$.

In the present paper,
from a viewpoint of asymptotic geometry,
we study problems of when a $\CAT(0)$ space is homeomorphic to $\R^n$.
In global Riemannian geometry,
many problems on the structure
of open Riemannian manifolds of non-negative Ricci curvature
has been studied in terms of their volume growths.
Cheeger--Colding \cite[Theorem A.1.11]{cheeger-colding} 
concluded that
if an $n$-dimensional open Riemannian manifold of non-negative Ricci curvature
has sufficiently large volume growth,
then it is diffeomorphic to $\R^n$.
In this paper,
we describe conditions for $\CAT(0)$ spaces 
of small volume growth to be homeomorphic to $\R^n$.

\subsection{Main results}

We denote by $\Haus^n$ the $n$-dimensional Hausdorff measure.
We denote by $\omega_0^n(r)$ 
the $n$-dimensional Hausdorff measure of a metric ball in $\R^n$
of radius $r$.
For a point $p$ in a metric space,
we denote by $U_r(p)$ the open metric ball of radius $r$
around $p$.
For a metric space $X$,
we define a non-negative value $\G_0^n(X) \in [0,\infty]$ by
\begin{equation}
\G_0^n(X) := \limsup_{t \to \infty} 
\frac{\Haus^n \left( U_t(p) \right)}{\omega_0^n(t)}
\label{eqn: evg}
\end{equation}
for some base point $p$ in $X$.
We remark that
$\G_0^n(X)$ does not depend on the choice of base points.
We call $\G_0^n(X)$ the 
(\emph{upper})
\emph{$n$-dimensional Euclidean volume growth of $X$}.

A metric space is said to be 
\emph{proper}
if every closed bounded subset is compact.
A geodesic metric space is
\emph{geodesically complete}
if every geodesic can be extended to a local geodesic defined on $\R$.
We notice that
if a $\CAT(0)$ space is geodesically complete,
then every geodesic can be extended to a geodesic line defined on $\R$.

Let $X$ be a proper, geodesically complete $\CAT(0)$ space.
We denote by $X^n$ the 
\emph{$n$-dimensional part of $X$}
determined as the set of all points in $X$ at which
all sufficiently small open metric balls 
have topological dimension $n$.
A relative volume comparison of Bishop--Gromov type
(Proposition \ref{prop: relvolcomp})
for $\CAT(0)$ spaces tells us that
if $X^n$ is non-empty,
then the following hold:
(1)
$\G_0^n(X) \ge 1$;
(2)
if $\G_0^n(X)$ is finite,
then for any $p \in X^n$ the limit superior in \eqref{eqn: evg} 
turns out to be the limit.

We say that a separable metric space is
\emph{purely $n$-dimensional}
if every non-empty open subset has topological dimension $n$.
A proper, geodesically complete $\CAT(0)$ space $X$
is purely $n$-dimensional 
if and only if $X = X^n$
(\cite[Theorem 1.2]{lytchak-nagano1}).

We prove the following asymptotic topological regularity:

\begin{thm}\label{thm: 1d}
For every $\epsilon \in (0,\infty)$,
and for every $n \in \N$,
there exists $\delta \in (0,\infty)$ 
satisfying the following property:
If a purely $n$-dimensional, 
proper, geodesically complete $\CAT(0)$ space
$X$ satisfies
$\G_0^n(X) < 1 + \delta$,
then $X$ is $(1+\epsilon)$-bi-Lipschitz homeomorphic to $\R^n$.
\end{thm}

In Theorem \ref{thm: 1d},
the pureness on the dimension 
is essential since we can construct counterexamples 
possessing lower dimensional subsets. 
For instance, for each $n \in \N$ with $n \ge 2$,
the one-point union $\R^n \vee \R$ of $\R^n$ and $\R$
equipped with the gluing metric is an $n$-dimensional, 
proper, geodesically complete $\CAT(0)$ space
with $\G_0^n(\R^n \vee \R) = 1$.

We describe the following optimal condition of small volume growth
for purely $n$-dimensional CAT(0) spaces 
to be homeomorphic to $\R^n$:

\begin{thm}\label{thm: 3/2}
If a purely $n$-dimensional, proper, 
geodesically complete $\CAT(0)$ space $X$ satisfies
$\G_0^n(X) < 3/2$,
then it is homeomorphic to $\R^n$.
\end{thm}

We denote by $T$ the discrete metric space consisting of three points
with pairwise distance $\pi$.
The condition of $\G_0^n(X)$ in Theorem \ref{thm: 3/2} is optimal
since for the $\ell^2$-product metric space 
$\R^{n-1} \times C_0(T)$
of $\R^{n-1}$ and the Euclidean cone $C_0(T)$ over $T$
we have
$\G_0^n \left( \R^{n-1} \times C_0(T) \right) = 3/2$.

We obtain the following characterization as the critical case:

\begin{thm}\label{thm: just3/2}
If a purely $n$-dimensional, proper, 
geodesically complete $\CAT(0)$ space $X$
satisfies
$\G_0^n(X) = 3/2$,
then $X$ is either homeomorphic to $\R^n$ or 
isometric to the $\ell^2$-product metric space 
$\R^{n-1} \times C_0(T)$,
where $C_0(T)$ is the Euclidean cone over $T$.
\end{thm}

Once we know Theorem \ref{thm: just3/2},
for homology manifolds
we can hope to relax the condition on $\G_0^n$
in Theorem \ref{thm: 3/2}. 
We notice that every complete $\CAT(0)$ homology manifold 
is proper and geodesically complete
(\cite[Corollary I.3.8]{bridson-haefliger}, 
\cite[Lemma 4.1]{lytchak-nagano1}, \cite[Lemma 3.1]{lytchak-nagano2}).

We prove the following regularity
of $\CAT(0)$ homology manifolds:

\begin{thm}\label{thm: 3/2d}
For every $n \in \N$,
there exists $\delta \in (0,\infty)$ satisfying the following property:
If a complete $\CAT(0)$ homology $n$-manifold $X$
satisfies
$\G_0^n(X) < 3/2 + \delta$,
then it is homeomorphic to $\R^n$.
\end{thm}
 
Thurston \cite[Theorem 3.3]{thurston} showed that
every homology $3$-manifold with an upper curvature bound
is a topological $3$-manifold.
Lytchak and the author \cite[Theorem 1.2]{lytchak-nagano2} 
proved that
for every homology $n$-manifold $M$ with an upper curvature bound
there exists a locally finite subset $E$ of $M$ such that
$M-E$ is a topological $n$-manifold.

\subsection{Relations with asymptotic geometry}

The Tits boundaries of $\CAT(0)$ spaces
have been utilized in asymptotic geometry
concerning flat subspaces.
For instance,
Kleiner--Leeb \cite{kleiner-leeb} employed the Tits boundaries
in their studies of rigidity of quasi-isometries
for symmetric spaces and Euclidean buildings.
Leeb \cite{leeb} described metric characterizations 
of symmetric spaces and Euclidean buildings 
in terms of their Tits boundaries.
Related subsequent studies on 
metric characterizations can be seen in
\cite{balser-lytchak}, \cite{hruska-kleiner}, \cite{lytchak3}, \cite{nagano1},
\cite{ricks1}, \cite{ricks2}, and so on.
We notice that every complete $\CAT(0)$ space $X$
admitting a geodesic ray
has the Tits boundary $\partial_{\T}X$
that is a complete $\CAT(1)$ space.

Let $X$ be a proper, geodesically complete $\CAT(0)$ space.
It seems to be well-known that
if $\partial_{\T}X$ is isometric to 
the $(n-1)$-dimensional standard unit sphere $\Sph^{n-1}$,
then $X$ is isometric to $\R^n$.
This rigidity follows from an observation of
Leeb \cite[Proposition 2.1]{leeb},
obtained as a generalization of Schroeder's work in 
\cite[Appendix 4]{ballmann-gromov-schroeder} for Hadamard manifolds.
Indeed,
Leeb \cite[Proposition 2.1]{leeb} showed that
for an arbitrary proper $\CAT(0)$ space,
if we find a subspace $\Sigma$ of the Tits boundary
such that 
$\Sigma$ is isometric to $\Sph^{n-1}$
and does not bound a unit hemisphere,
then there exists an $n$-dimensional flat subspace $\Pi$ 
with $\partial_{\T}\Pi = \Sigma$.

As a result of asymptotic geometric regularity,
we prove that if $\partial_{\T}X$ is sufficiently close to $\Sph^{n-1}$
with respect to the Gromov--Hausdorff distance,
then $X$ is bi-Lipschitz homeomorphic to $\R^n$
(see Theorems \ref{thm: ageomreg0} and \ref{thm: ageomreg}).
A desired bi-Lipschitz homeomorphism 
is given by a map on $X$ with Busemann function coordinates.
When we analyze such a regular map on $X$ 
with Busemann function coordinates,
we utilize the ideas of the theory 
of strainer maps on $\GCBA$ spaces 
with distance function coordinates
developed by 
Lytchak and the author \cite{lytchak-nagano1}.

We emphasize that $\partial_{\T}X$ is not necessarily compact,
and it is not necessarily geodesically complete,
although $X$ is proper and geodesically complete.
We show that if $X$ has the Gromov--Hausdorff asymptotic cone $C_{\infty}X$,
then $C_{\infty}X$ is isometric to the Euclidean cone over $\partial_{\T}X$;
in particular,
$\partial_{\T}X$ is compact and geodesically complete
(Proposition \ref{prop: accat}).
We also prove that 
a purely $n$-dimensional, proper, 
geodesically complete $\CAT(0)$ space 
is doubling
if and only if it has the Gromov--Hausdorff asymptotic cone,
and if and only if it has finite $n$-dimensional Euclidean volume growth
(Proposition \ref{prop: chaccat}).
These observations enable us to prove our main results
of asymptotic topological regularity.

\subsection{Outlines of the proofs of the main results}

Let $X$ be a purely $n$-dimensional, proper, 
geodesically complete $\CAT(0)$ space.
Suppose that $X$ has finite $n$-dimensional Euclidean volume growth.
Then 
$X$ has the Gromov--Hausdorff asymptotic cone 
isometric to the Euclidean cone $C_0(\partial_{\T}X)$
over the Tits boundary $\partial_{\T}X$;
in particular,
$\partial_{\T}X$ is compact, geodesically complete,
and purely $(n-1)$-dimensional;
moreover, 
\begin{equation}
\G_0^n(X) = 
\frac{\Haus^{n-1} \left( \partial_{\T}X \right)}
{\Haus^{n-1} \left( \Sph^{n-1} \right)}
\label{eqn: limit-ratio}
\end{equation}
(Propositions \ref{prop: accat} and \ref{prop: chaccat}).
According to the upper bounds for $\G_0^n(X)$
in Theorems \ref{thm: 1d}--\ref{thm: 3/2d},
the Tits boundary $\partial_{\T}X$ is characterized as follows:
\begin{enumerate}
\item
If we have $\G_0^n(X) < 1+\delta$ for sufficiently small $\delta$,
then a volume rigidity result of the author \cite[Theorem 1.10]{nagano2}
for $\CAT(1)$ spaces
implies that $\partial_{\T}X$ is bi-Lipschitz homeomorphic to $\Sph^{n-1}$.
\item
If we have $\G_0^n(X) < 3/2$,
then a volume sphere theorem 
of Lytchak and the author \cite[Theorem 8.3]{lytchak-nagano2}
for $\CAT(1)$ spaces
implies that
$\partial_{\T}X$ is homeomorphic to $\Sph^{n-1}$.
\item
If we have $\G_0^n(X) = 3/2$,
then a characterization of the author \cite[Theorem 1.1]{nagano4}
for $\CAT(1)$ spaces 
implies that $\partial_{\T}X$ is either homeomorphic to $\Sph^{n-1}$ or
the spherical suspension $\Sph^{n-2} \ast T$.
\item
If we have $\G_0^n(X) < 3/2 + \delta$ for sufficiently small $\delta$,
and if $X$ is a homology $n$-manifold,
then $\partial_{\T}X$ is a homology $(n-1)$-manifold,
and hence
a volume sphere theorem of the author \cite[Theorem 1.2]{nagano4}
for $\CAT(1)$ homology manifolds
implies that $\partial_{\T}X$ is homeomorphic to $\Sph^{n-1}$.
\end{enumerate}

From properties (1)--(4) listed above, 
we can derive Theorems \ref{thm: 1d}--\ref{thm: 3/2d}.
In the proofs of 
Theorems \ref{thm: 3/2}, \ref{thm: just3/2}, and \ref{thm: 3/2d},
when we prove that $X$ is a topological $n$-manifold,
we use the local topological regularity theorem 
of Lytchak and the author \cite[Theorem 1.1]{lytchak-nagano2}.
In the proof of Theorem \ref{thm: just3/2},
in order to determine the geometric structure of $X$,
we describe a volume regularity condition for $\CAT(1)$ spaces
to be almost isometric to a compact spherical building
(Proposition \ref{prop: volregcat1}).

\subsection{Organization}

In Section 2,
we discuss basic concepts in the geometry 
of metric spaces with curvature bounded above.

In Section 3,
we study relations between Gromov--Hausdorff asymptotic cones of
$\CAT(0)$ spaces and their Tits boundaries,
and we show the observations 
mentioned in Subsection 1.2.

In Section 4,
we recall the basic properties of strainer maps discussed 
by Lytchak and the author \cite{lytchak-nagano1}.
Using the homotopic stability theorem 
\cite[Theorem 13.1]{lytchak-nagano1}
of fibers of strainer maps,
we prove that if a proper, geodesically complete $\CAT(0)$ space $X$
has the Gromov--Hausdorff asymptotic cone,
then any sufficiently large metric sphere in $X$
is homotopy equivalent to $\partial_{\T}X$
(Theorem \ref{thm: hstabinf});
in particular,
if in addition $X$ is a topological $n$-manifold,
and if $\partial_{\T}X$ is simply connected,
then it is homeomorphic to $\R^n$
(Theorem \ref{thm: sci}).

In Section 5,
for a proper, geodesically complete $\CAT(0)$ space $X$,
we study a map $F \colon X \to \R^m$
with Busemann function coordinates
satisfying a regular property,
called a 
\emph{Busemann strainer map}.
We prove that
if $X$ is $n$-dimensional,
then a Busemann strainer map $F \colon X \to \R^n$
becomes a bi-Lipschitz homeomorphism
(Propositions \ref{prop: bilipbsm0} and \ref{prop: bilipbsm}).
We also show the result of asymptotic geometric regularity
(Theorems \ref{thm: ageomreg0} and \ref{thm: ageomreg})
mentioned in Subsection 1.2.

In Section 6,
we discuss the proofs of Theorems \ref{thm: 1d}--\ref{thm: 3/2d}.

\subsection{Problem}

As a natural question beyond Theorem \ref{thm: 3/2d},
we pose the following asymptotic regularity problem
for $\CAT(0)$ spaces:

\begin{prob}\label{prob: ar-problem}
Let $n \in \N$ satisfy $n \ge 4$.
Let $c_n$ be the supremum of $c \in (3/2,\infty)$
for which every complete $\CAT(0)$ homology $n$-manifold $X$
with
$\G_0^n(X) \le c$
is homeomorphic to $\R^n$.
\begin{enumerate}
\item
Find the concrete value $c_n$.
\item
Describe all complete $\CAT(0)$ homology $n$-manifolds $X$
satisfying $\G_0^n(X) = c_n$
in the maximal critical case.
\end{enumerate}
\end{prob}

This problem is closely related to
a volume pinching problem posed in \cite[Problem 1.1]{nagano4}
for $\CAT(1)$ spaces
(cf.~\eqref{eqn: limit-ratio}).

\subsection*{Acknowledgments}

The author would like to express his gratitude to
Alexander Lytchak for valuable discussions 
and helpful comments in private communications.
The author would like to thank
Takashi Shioya
for his interest in this work.
The author would also like to thank the referees 
for carefully reading the manuscript and for giving helpful comments.


\section{Preliminaries}

We refer the readers to 
\cite{alexander-kapovitch-petrunin-0},
\cite{alexander-kapovitch-petrunin}, \cite{alexandrov-berestovskii-nikolaev},
\cite{ballmann}, \cite{bridson-haefliger}, \cite{burago-burago-ivanov},
\cite{buyalo-schroeder} 
for the basic facts
on metric spaces with an upper curvature bound.

\subsection{Metric spaces}

Let $r \in (0,\infty)$.
For a point $p$ in a metric space,
we denote by $U_r(p)$, $B_r(p)$, and $S_r(p)$
the open metric ball of radius $r$ around $p$,
the closed one, and the metric sphere, respectively.

Let $X$ be a metric space.
A subset $A$ of $X$ is called an
\emph{$r$-net of $X$}
if $\bigcup_{p \in A} U_r(p)$ coincides with $X$.
A subset $A$ of $X$ is said to be 
\emph{$r$-separated}
if every pair of points in $A$ has distance at least $r$.
Due to Zorn's lemma,
for every subset $W$ of $X$,
for every $s \in (0,\infty)$
there exists a maximal $s$-separated set $A$ in $W$;
in this case, $A$ is an $s$-net of $W$.

For $N \in \N$,
a metric space $X$ is said to be 
\emph{$N$-doubling}
if every open metric ball of radius $r$ in $X$ can be covered 
by at most $N$ open metric balls of radius $r/2$ in $X$.
A metric space $X$ is 
\emph{doubling}
if $X$ is $N$-doubling for some $N$.
If a metric space $X$ is $N$-doubling,
then so is every metric subspace of $X$.
A metric space $X$ is doubling
if and only if
there exists some $N \in \N$
such that
for each $s \in (0,\infty)$, and for each $x \in X$,
every $s$-separated set in $U_{2s}(x)$
has at most $N$ elements.

For a metric space $X$ with metric $d_X$,
and for a positive number $\lambda \in (0,\infty)$,
we denote by $\lambda X$ the rescaled metric space
defined as $(X, \lambda d_X)$.
If a metric space $X$ is $N$-doubling,
then for every $\lambda \in (0,\infty)$
the rescaled metric space $\lambda X$ is also $N$-doubling.

Let $X$ be a metric space with metric $d_X$.
Let $d_X \wedge \pi$ be the $\pi$-truncated metric on $X$
defined by $d_X \wedge \pi := \min \{ d_X, \pi \}$.
The 
\emph{Euclidean cone $C_0(X)$ over $X$}
is defined as the cone 
$[0,\infty) \times X / \{ 0 \} \times X$
over $X$
equipped with the Euclidean metric
$d_{C_0(X)}$ given by
\[
d_{C_0(X)} \left( [(t_1,x_1)], [(t_2,x_2)] \right)^2 := 
t_1^2 + t_2^2 - 2t_1t_2 \cos \left( (d_X \wedge \pi) (x_1,x_2) \right).
\]
For simplicity,
we write an element $[(t,x)]$ in $C_0(X)$ as $tx$,
and denote by $0$ the vertex of $C_0(X)$.
For metric spaces $Y$ and $Z$,
we denote by $Y \ast Z$ the spherical join of $Y$ and $Z$.
Note that $C_0(Y \ast Z)$ is isometric to 
the $\ell^2$-direct product metric space
$C_0(Y) \times C_0(Z)$
of $C_0(Y)$ and $C_0(Z)$.

\subsection{Maps between metric spaces}

Let $c \in (0,\infty)$.
Let $X$ be a metric space with metric $d_X$,
and $Y$ a metric space with metric $d_Y$.
A map $f \colon X \to Y$ is said to be
\emph{$c$-Lipschitz} 
if $d_Y(f(x_1),f(x_2)) \le cd_X(x_1,x_2)$
for all $x_1, x_2 \in X$.
A map $f \colon X \to Y$ is said to be
\emph{$c$-bi-Lipschitz} 
if $f$ is $c$-Lipschitz,
and if $d_X(x_1,x_2) \le c d_Y(f(x_1),f(x_2))$
for all $x_1, x_2 \in X$
(consequently, $c \in [1,\infty)$).
A $1$-bi-Lipschitz homeomorphism
is nothing but an isometry,
and a $1$-bi-Lipschitz embedding is 
an isometric embedding.
A map $f \colon X \to Y$ is 
\emph{$c$-open}
if for any $r \in (0,\infty)$ and $x \in X$
such that $B_{cr}(x)$ is complete,
the ball $U_r(f(x))$ in $Y$ is contained in the image $f(U_{cr}(x))$
of the ball $U_{cr}(x)$ in $X$.
In the case where $X$ is complete,
if a map $f \colon X \to Y$ is $c$-open,
then $f$ is surjective;
indeed,
for every $x \in X$, and for every $y \in Y$ with $y \neq f(x)$,
by setting $t := d_Y(f(x),y)$,
we find $x_0 \in U_{2ct}(x)$ with $y = f(x_0)$
since the ball $B_{2ct}(x)$ in $X$ is complete,
and hence the ball $U_{2t}(f(x))$ in $Y$ is contained in 
the image $f(U_{2ct}(x))$.
Moreover,
in the case where $X$ is complete,
a $c$-Lipschitz map $f \colon X \to Y$
is a $c$-bi-Lipschitz homeomorphism for some $c \in [1,\infty)$
if and only if $f$ is an injective $c$-open map.

A map $\varphi \colon X \to Y$ is said to be a
\emph{$c$-approximation} between $X$ and $Y$
if $\varphi(X)$ is a $c$-net of $Y$,
and if for all $x_1, x_2 \in X$ we have
\[
\left\vert
d_Y \left( \varphi(x_1), \varphi(x_2) \right) - d_X(x_1,x_2)
\right\vert
< c.
\]
If there exists a $c$-approximation $\varphi \colon X \to Y$,
then there exists a $2c$-approximation $\psi \colon Y \to X$ such that
for all $x \in X$ and $y \in Y$ we have
$d_X \left( (\psi \circ \varphi)(x), x \right) < 2c$ and
$d_Y \left( (\varphi \circ \psi)(y), y \right) < 2c$.

\subsection{Geodesic metric spaces}

Let $X$ be a metric space.
A \emph{geodesic $\gamma \colon I \to X$} 
means an isometric embedding from an interval $I$.
For a pair of points $p, q$ in $X$,
a \emph{geodesic $pq$ in $X$ from $p$ to $q$}
means the image of an isometric embedding
$\gamma \colon [a,b] \to X$
from a bounded closed interval $[a,b]$
with $\gamma(a) = p$ and $\gamma(b) = q$.
A geodesic $\gamma \colon I \to X$ is called 
a \emph{ray} 
if $I = [0,\infty)$,
and $\gamma$ is called a \emph{line} if $I = \R$.

For $r \in (0,\infty]$,
a metric space $X$ is said to be
\emph{$r$-geodesic}
if every pair of points in $X$ with distance smaller than $r$
can be joined by a geodesic in $X$.
A metric space is 
\emph{geodesic}
if it is $\infty$-geodesic.
A geodesic metric space is proper
if and only if
it is complete and locally compact.

For $r \in (0,\infty]$,
a subset $C$ of a metric space is said to be 
\emph{$r$-convex}
if $C$ itself is $r$-geodesic as a metric subspace, 
and if every geodesic joining two points in $C$
is contained in $C$.
A subset $C$ of a metric space is 
\emph{convex}
if $C$ is $\infty$-convex.

\subsection{Gromov--Hausdorff topology}

We denote by $d_{\GH}$ the Gromov--Hausdorff distance 
between metric spaces.
If for $c \in (0,\infty)$ two metric spaces $X$ and $Y$
satisfy $d_{\GH}(X,Y) < c$,
then there exists a $2c$-approximation between $X$ and $Y$.
We say that a sequence $(X_i)$ of metric spaces 
converges to a metric space $X$ 
\emph{in the Gromov--Hausdorff topology}
if $\lim_{i \to \infty} d_{\GH}(X_i,X) = 0$.
Due to the Gromov precompactness theorem,
for fixed $r \in (0,\infty)$ and $N \in \N$,
every sequence of 
$N$-doubling compact metric spaces of diameter at most $r$
has a Gromov--Hausdorff convergent subsequence
whose limit is $N$-doubling.

We say that a sequence $(X_i,p_i)$ of pointed geodesic metric spaces 
converges to a pointed metric space $(X,p)$
\emph{in the pointed Gromov--Hausdorff topology}
if for every $r \in (0,\infty)$
there exists a sequence $(\epsilon_i)$ in $(0,\infty)$
with $\lim_{i \to \infty}\epsilon_i = 0$
such that for each $i$
there exists a
$\epsilon_i$-approximation $\varphi_i \colon B_r(p) \to B_r(p_i)$
with $\varphi_i(p) = p_i$;
in this case,
we write $(X,p) = \lim_{i \to \infty} (X_i,p_i)$.
If a sequence $(X_i,p_i)$ of pointed proper geodesic metric spaces 
converges to a metric space $(X,p)$
in the pointed Gromov--Hausdorff topology,
then $X$ is proper and geodesic.

\subsection{CAT$\boldsymbol{(\kappa)}$ spaces}

For $\kappa \in \R$,
we denote by $M_{\kappa}^n$ 
the simply connected, complete Riemannian $n$-manifold 
of constant curvature $\kappa$,
and denote by $D_{\kappa}$ the diameter of $M_{\kappa}^n$.
A metric space $X$ is said to be 
$\CAT(\kappa)$
if $X$ is $D_{\kappa}$-geodesic,
and if every geodesic triangle in $X$ with perimeter 
smaller than $2D_{\kappa}$
is not thicker than the comparison triangle 
with the same side lengths in $M_{\kappa}^2$.

Let $X$ be a $\CAT(\kappa)$ space.
Every pair of points in $X$ with distance smaller than $D_{\kappa}$
can be uniquely joined by a geodesic.
Let $p \in X$.
For every $r \in (0,D_{\kappa}/2]$,
the balls $U_r(p)$ and $B_r(p)$ are convex.
Along the geodesics emanating from $p$,
for every $r \in (0,D_{\kappa})$
the balls $U_r(p)$ and $B_r(p)$ are contractible inside themselves.
Every open subset of $X$ is an $\ANR$ 
(\cite{ontaneda}, \cite{kramer}).
For $x, y \in U_{D_{\kappa}}(p) - \{p\}$,
we denote by $\angle_p(x,y)$
the angle at $p$
between $px$ and $py$.
Put $\Sigma_p'X := \{ \, px \mid x \in U_{D_{\kappa}}(p) - \{p\} \, \}$.
The angle $\angle_p$ at $p$
is a pseudo-metric on $\Sigma_p'X$.
The 
\emph{space of directions $\Sigma_pX$ at $p$}
is defined as the $\angle_p$-completion of
the quotient metric space $\Sigma_p'X / \angle_p = 0$.
For $x \in U_{D_{\kappa}}(p) - \{p\}$,
we denote by $x_p' \in \Sigma_pX$
the starting direction of $px$ at $p$.
The 
\emph{tangent space $T_pX$ at $p$} 
is defined as $C_0(\Sigma_pX)$.
The space $\Sigma_pX$ is $\CAT(1)$,
and the space $T_pX$ is $\CAT(0)$.
In fact,
for a metric space $\Sigma$,
the Euclidean cone $C_0(\Sigma)$ 
is $\CAT(0)$ if and only if $\Sigma$ is $\CAT(1)$.
For two metric spaces $Y$ and $Z$,
the spherical join $Y \ast Z$ is $\CAT(1)$
if and only if $Y$ and $Z$ are $\CAT(1)$.

\subsection{Ideal boundaries of \boldmath$\CAT(0)$ spaces}

Let $X$ be a metric space with metric $d_X$.
Two rays $\gamma_1, \gamma_2 \colon [0,\infty) \to X$
are said to be
\emph{asymptotic}
if $\sup_{t \in [0,\infty)} d_X(\gamma_1(t),\gamma_2(t))$
is finite.
The asymptotic relation gives an equivalence relation
on the set of all rays in $X$.
The 
\emph{ideal boundary $\partial_{\infty}X$ of $X$}
is defined as the set of all asymptotic equivalence classes of rays in $X$.
For a ray $\gamma$ in $X$,
we denote by $\gamma(\infty)$
the asymptotic equivalent class of $\gamma$ in $\partial_{\infty}X$.

Let $X$ be a complete $\CAT(0)$ space.
For every $p \in X$, and for every $\xi \in \partial_{\infty}X$,
there exists a unique ray $\gamma \colon [0,\infty) \to X$
with $\gamma(0) = p$ and $\gamma(\infty) = \xi$.
For $p \in X$ and $\xi \in \partial_{\infty}X$,
we denote by $\gamma_{p\xi}$
the unique ray emanating from $p$ to $\xi$,
by $p\xi$ the image of $\gamma_{p\xi}$,
and by $\xi_p' \in \Sigma_pX$ the starting direction of $p\xi$ at $p$.
For $p \in X$, and $\xi, \eta \in \partial_{\infty}X$,
we denote by $\angle_p(\xi,\eta)$
the angle at $p$
between $p\xi$ and $p\eta$.
The 
\emph{angle metric $\angle$ on $\partial_{\infty}X$}
is defined by
$\angle(\xi,\eta) := \sup_{p \in X} 
\angle_p(\xi,\eta)$.
The 
\emph{Tits metric $d_{\T}$ on $\partial_{\infty}X$}
is defined as the length metric on $\partial_{\infty}X$ 
induced from $\angle$.
Notice that $\angle = \min \{ d_{\T}, \pi \}$,
and $d_{\T}$ possibly takes the value $\infty$.
We denote by $\partial_{\T}X$
the ideal boundary $\partial_{\infty}X$
equipped with the Tits metric $d_{\T}$,
and call it the
\emph{Tits boundary of $X$}.
Then $\partial_{\T}X$ is a complete $\CAT(1)$ space.
The Euclidean cone $C_0(\partial_{\T}X)$ 
is isometric to the Euclidean cone $C_0(\partial_{\infty}X)$ over 
the ideal boundary $\partial_{\infty}X$ with the angle metric $\angle$.

We recall the following basic asymptotic property
(see e.g., 
\cite[Proposition II.9.8, Corollary II.9.10]{bridson-haefliger}):

\begin{lem}\label{lem: scalingcat0}
Let $X$ be a complete $\CAT(0)$ space with metric $d_X$.
Let $p \in X$.
Then for all $\xi_1, \xi_2 \in \partial_{\T}X$,
and for all $a_1, a_2 \in (0,\infty)$,
we have
\[
d_{C_0(\partial_{\T}X)} \left( a_1\xi_1, a_2\xi_2 \right)
= \lim_{t \to \infty}
\frac{d_X \left( \gamma_{p\xi_1} \left( a_1t \right), 
\gamma_{p\xi_2} \left( a_2t \right) \right)}{t},
\]
as the monotone non-decreasing limit.
Moreover,
for every $\lambda \in (0,\infty)$
the map $\Phi_p^{\lambda} \colon C_0(\partial_{\T}X) \to \lambda X$
defined by 
$\Phi_p^{\lambda} \left( a\xi \right) 
:= \gamma_{p\xi} \left( a/\lambda \right)$
is $1$-Lipschitz,
where $\gamma_{p\xi} \colon [0,\infty) \to X$ 
is the ray in $X$ from $p$ to $\xi$.
\end{lem}

We recall the following splitting theorem
for $\CAT(0)$ spaces
(\cite[Theorem II.9.24]{bridson-haefliger},
and \cite[Appendix 4]{ballmann-gromov-schroeder} 
in the Riemannian setting).

\begin{prop}\label{prop: tits-splitting}
\emph{(\cite{bridson-haefliger})} 
Let $X$ be a geodesically complete, complete $\CAT(0)$ space.
If $\partial_{\T}X$ isometrically splits as 
a spherical join $\Sigma_1 \ast \Sigma_2$,
then $X$ is isometric to an 
$\ell^2$-direct product metric space $X_1 \times X_2$
of geodesically complete, complete $\CAT(0)$ spaces $X_1$ and $X_2$
such that for each $j \in \{ 1, 2\}$ 
the space $\partial_{\T}X_j$ is isometric to $\Sigma_j$.
\end{prop}

\subsection{Geodesically complete CAT$\boldsymbol{(\kappa)}$ spaces}

We refer the readers to \cite{lytchak-nagano1} for 
the basic properties of $\GCBA$ spaces, that is,
locally compact, separable, locally geodesically complete
metric spaces with an upper curvature bound.
Recall that
a $\CAT(\kappa)$ space is said to be 
\emph{locally geodesically complete}
(or has \emph{geodesic extension property})
if every geodesic defined on a compact interval can be extended to
a local geodesic beyond endpoints.
A $\CAT(\kappa)$ space is 
\emph{geodesically complete}
if every geodesic can be extended to a local geodesic defined on $\R$.
Every locally geodesically complete, complete $\CAT(\kappa)$ space
is geodesically complete.
The geodesical completeness for compact (resp.~proper) 
$\CAT(\kappa)$ spaces
is preserved under the (resp.~pointed) Gromov--Hausdorff limit.

Let $X$ be a proper, geodesically complete $\CAT(\kappa)$ space.
For every $p \in X$,
the space $\Sigma_pX$ is compact and geodesically complete,
and $T_pX$ is proper and geodesically complete.
In fact,
for a $\CAT(1)$ space $\Sigma$,
the Euclidean cone $C_0(\Sigma)$ is geodesically complete
if and only if $\Sigma$ is geodesically complete
and not a singleton.
For two $\CAT(1)$ spaces $Y$ and $Z$,
the spherical join $Y \ast Z$ is geodesically complete
if and only if
$Y$ and $Z$ are geodesically complete
and not a singleton.

\subsection{Dimension of CAT$\boldsymbol{(\kappa)}$ spaces}

Let $X$ be a separable $\CAT(\kappa)$ space.
The covering (topological) dimension $\dim X$ satisfies
\[
\dim X = 1 + \sup_{p \in X} \dim \Sigma_pX
= \sup_{p \in X} \dim T_pX
\]
(\cite{kleiner}).
Assume in addition that
$X$ is proper and geodesically complete.
Every relatively compact open subset of $X$
has finite covering dimension
(see \cite[Subsection 5.3]{lytchak-nagano1}).
The covering dimension $\dim X$ 
is equal to the Hausdorff dimension of $X$,
and equal to the supremum of $m$
such that $X$ has an open subset $U$ homeomorphic to $\R^m$
(\cite[Theorem 1.1]{lytchak-nagano1}).

From the studies in \cite[Subsection 11.3]{lytchak-nagano1}
on the stability of dimension,
we can derive the following (see \cite[Lemmas 2.1 and 2.3]{nagano4}):

\begin{lem}\label{lem: npure}
Let $(X_i,p_i)$ be a sequence
of pointed proper geodesically complete $\CAT(\kappa)$ spaces.
Assume that
$(X_i,p_i)$ converges to a pointed metric space $(X,p)$
in the pointed Gromov--Hausdorff topology.
Then we have
\[
\dim X \le
\liminf_{i \to \infty} \dim X_i.
\]
If in addition each $X_i$ is purely $n$-dimensional,
then so is $X$.
\end{lem}

On the Gromov--Hausdorff topology, 
we have the following continuity (see \cite[Lemma 2.2]{nagano4}):

\begin{lem}\label{lem: stab}
Let $(X_i)$ be a sequence
of compact geodesically complete $\CAT(\kappa)$ spaces.
Assume that $(X_i)$ converges to a metric space $X$
in the Gromov--Hausdorff topology.
Then we have
\[
\lim_{i \to \infty} \dim X_i = \dim X.
\]
\end{lem}

We say that a separable metric space is
\emph{pure-dimensional}
if it is purely $n$-dimensional for some $n$.

We have the following characterization 
(\cite[Proposition 8.1]{lytchak-nagano2}):

\begin{prop}\label{prop: pure}
\emph{(\cite{lytchak-nagano2})}
Let $X$ be a proper, geodesically complete,
geodesic $\CAT(\kappa)$ space.
Let $W$ be a connected open subset of $X$.
Then the following are equivalent:
\begin{enumerate}
\item
$W$ is pure-dimensional;
\item
for every $p \in W$ the space $\Sigma_pX$ is pure-dimensional;
\item
for every $p \in W$ the space $T_pX$ is pure-dimensional.
\end{enumerate}
\end{prop}

\subsection{On the volumes of CAT$\boldsymbol{(\kappa)}$ spaces}

For $\kappa \in \R$ and $r \in (0,D_{\kappa}]$,
we denote by $\omega_{\kappa}^n(r)$ 
the $n$-dimensional Hausdorff measure of any metric ball in $M_{\kappa}^n$
of radius $r$ if $n \ge 2$,
and by $\omega_{\kappa}^1(r)$ 
the $1$-dimensional Hausdorff measure of $[-r,r]$.

We recall that for every proper, geodesically complete 
$\CAT(\kappa)$ space $X$,
the $n$-dimensional part $X^n$ of $X$
coincides with the set of all points $p \in X$ with $\dim \Sigma_pX = n-1$;
if in addition $\dim X = n$,
then $X^n$ also coincides with 
the support of $\Haus^n$
(\cite[Theorem 1.2]{lytchak-nagano1}). 
For $\CAT(\kappa)$ spaces,
we have the following absolute volume comparison 
of Bishop--G\"{u}nther type 
(\cite[Proposition 6.1]{nagano2},
\cite[Proposition 3.1]{nagano4}):

\begin{prop}\label{prop: abvolcomp}
\emph{(\cite{nagano2})} \
Let $X$ be a proper, geodesically complete $\CAT(\kappa)$ space,
and let $p \in X$ be a point with $p \in X^n$.
Then for every $r \in (0,D_{\kappa})$ we have
\[
\Haus^n \left( U_r(p) \right) \ge \omega_{\kappa}^n(r).
\]
Moreover,
if in addition $X$ is purely $n$-dimensional,
then the equality holds if and only if
the pair $(U_r(p),p)$ is isometric to $(U_r(\tilde{p}),\tilde{p})$
for any point $\tilde{p} \in M_{\kappa}^n$.
\end{prop}

Furthermore,
we have the following relative volume comparison
of Bishop--Gromov type 
(\cite[Proposition 6.3]{nagano2},
\cite[Proposition 3.2]{nagano4}):

\begin{prop}\label{prop: relvolcomp}
\emph{(\cite{nagano2})} \ \
Let $X$ be a proper, geodesically complete $\CAT(\kappa)$ space,
and let $p \in X$ be a point with $p \in X^n$.
Then the function $f_p \colon (0,D_{\kappa}) \to [1,\infty]$ defined by
\[
f_p(t) := \frac{\Haus^n \left( U_t(p) \right)}{\omega_{\kappa}^n(t)}
\]
is monotone non-decreasing.
\end{prop}

Let $(X_i)$
be a sequence of compact geodesically complete 
$\CAT(\kappa)$ spaces of $\dim X_i = n$
converging to a metric space $X$
in the Gromov--Hausdorff topology.
By Lemma \ref{lem: stab},
the compact, geodesically complete $\CAT(\kappa)$ space 
$X$ satisfies $\dim X = n$.
Let $(X_i,p_i)$
be a sequence of pointed, proper geodesically complete 
$\CAT(\kappa)$ spaces of $\dim X_i = n$
converging to a pointed metric space $(X,p)$
in the pointed Gromov--Hausdorff topology.
By Lemma \ref{lem: npure},
the proper, geodesically complete $\CAT(\kappa)$ space 
$X$ satisfies $\dim X \le n$.

We quote the volume convergence theorem for $\CAT(\kappa)$ spaces
in \cite[Theorem 1.1]{nagano2} in the following form:

\begin{thm}\label{thm: volconv}
\emph{(\cite{nagano2})}
If a sequence $(X_i)$ of $n$-dimensional, compact geodesically complete 
$\CAT(\kappa)$ spaces converges to some metric space $X$
in the Gromov--Hausdorff topology,
then we have
\[
\Haus^n \left( X \right) =
\lim_{i \to \infty} \Haus^n \left( X_i \right).
\]
If a sequence $(X_i,p_i)$
of pointed, $n$-dimensional, proper geodesically complete 
$\CAT(\kappa)$ spaces
converges to some pointed metric space $(X,p)$
in the pointed Gromov--Hausdorff topology,
then for every $r \in (0,\infty)$ with $\dim U_r(p) = n$ we have
\[
\Haus^n \left( U_r(p) \right) =
\lim_{i \to \infty} \Haus^n \left( U_r(p_i) \right).
\]
\end{thm}

The second half of Theorem \ref{thm: volconv} 
on the pointed Gromov--Hausdorff convergence,
not shown explicitly in \cite[Theorem 1.1]{nagano2},
can be proved by a similar argument to that discussed in 
\cite[Section 4]{nagano2}.

\begin{rem}\label{rem: ainfvolconv}
Lytchak and the author \cite{lytchak-nagano1}
introduced a positive Radon measure on an arbitrary $\GCBA$ space,
called the 
\emph{canonical measure},
whose restriction to the $m$-dimensional part
coincides with $\Haus^m$
(see \cite[Theorem 1.4]{lytchak-nagano1}).
As a generalization of Theorem \ref{thm: volconv},
\cite{lytchak-nagano1} proved the continuity of the canonical measure
with respect to the Gromov--Hausdorff topology
for a sequence of compact $\GCBA$ spaces
of dimension, curvature and diameter bounded from above
and injectivity radius bounded below by some constants
\cite[Theorem 1.5]{lytchak-nagano1},
and general local statements
\cite[Section 12]{lytchak-nagano1}.
Cavallucci and Sambusetti \cite{cavallucci-sambusetti}
show the upper semi-continuity 
of the canonical measure of balls
with respect to the pointed Gromov--Hausdorff topology
for a sequence of proper $\GCBA$ spaces
\cite[Lemma 2.7]{cavallucci-sambusetti},
and formulate the continuity under a uniform measure doubling condition
\cite[Corollary 5.7]{cavallucci-sambusetti}.
\end{rem}

We next recall the following volume regularity
(\cite[Theorem 1.10]{nagano2}):

\begin{thm}\label{thm: volreg}
\emph{(\cite{nagano2})}
For every $\epsilon \in (0,\infty)$,
and for every $m \in \N$, 
there exists $\delta \in (0,\infty)$ 
satisfying the following property:
If a purely $m$-dimensional,
compact, geodesically complete $\CAT(1)$ space $\Sigma$ satisfies
\[
\Haus^m \left( \Sigma \right) < \Haus^m \left( \Sph^m \right) + \delta,
\]
then $\Sigma$ is $(1+\epsilon)$-bi-Lipschitz homeomorphic to $\Sph^m$.
\end{thm}


\section{Gromov--Hausdorff asymptotic cones}

In this section,
we denote by $O_{\infty}$
the set of all sequences $(\lambda_i)$ in $(0,\infty)$
with $\lim_{i \to \infty} \lambda_i = 0$.
We discuss basic properties
of Gromov--Hausdorff asymptotic cones of $\CAT(0)$ spaces.

\subsection{Asymptotic cones and Tits boundaries}

Let $X$ be a proper geodesic metric space.
Let $p \in X$.
For a sequence $(\lambda_i) \in O_{\infty}$,
the 
\emph{Gromov--Hausdorff asymptotic cone 
$C_{\infty}^{(\lambda_i)}X$ of $X$ 
with scale $(\lambda_i)$} 
is defined by
\[
\left( C_{\infty}^{(\lambda_i)}X, p_{\infty} \right) := \lim_{i \to \infty} 
\left( \lambda_i X, p \right),
\]
if the pointed Gromov--Hausdorff limit 
$\lim_{i \to \infty} \left( \lambda_i X, p \right)$ exists,
where $p_{\infty}$
is called the 
\emph{limit base point of $p$}.
Notice that $C_{\infty}^{(\lambda_i)}X$ does not depend on the choice
of the base point $p$,
and $C_{\infty}^{(\lambda_i)}X$ is proper.

The following lemma seems to be well-known
as a consequence of the Gromov precompactness theorem:

\begin{lem}\label{lem: acdoubling}
If a proper geodesic metric space $X$ is $N$-doubling,
then for some $(\lambda_i) \in O_{\infty}$
there exists the Gromov--Hausdorff asymptotic cone 
$C_{\infty}^{(\lambda_i)}X$ of $X$ 
with scale $(\lambda_i)$
such that 
$C_{\infty}^{(\lambda_i)}X$ is $N$-doubling.
\end{lem}

\begin{proof}
Take $p \in X$ and $r \in (0,\infty)$.
Let $(\lambda_i) \in O_{\infty}$.
Then the compact ball $B_r(p)$ in $\lambda_i X$ is $N$-doubling for all $i$.
From the Gromov precompactness theorem we deduce that
the pointed sequence $(\lambda_i X, p)$ has a Gromov--Hausdorff
convergent subsequence whose limit is $N$-doubling.
Discussing diagonal arguments adequately,
we see that $X$ has an $N$-doubling 
Gromov--Hausdorff asymptotic cone with some scale.
\end{proof}

For a proper $\CAT(0)$ space,
the existence of a Gromov--Hausdorff asymptotic cone with some scale
leads to the compactness of the Tits boundary.
Namely, we have the following:

\begin{lem}\label{lem: conecpt}
Let $X$ be a proper $\CAT(0)$ space 
whose Tits boundary $\partial_{\T}X$ is non-empty.
If $X$ has a Gromov--Hausdorff asymptotic cone 
$C_{\infty}^{(\lambda_i)}X$ of $X$ 
with some scale $(\lambda_i)$,
then $\partial_{\T}X$ is compact.
\end{lem}

\begin{proof}
Assume that 
there exists a Gromov--Hausdorff asymptotic cone 
$C_{\infty}^{(\lambda_i)}X$ of $X$ 
with some scale $(\lambda_i)$.
Then for a fixed base point $p \in X$
the pointed sequence $(\lambda_i X, p)$
converges to 
$\left( C_{\infty}^{(\lambda_i)}X, p_{\infty} \right)$
in the pointed Gromov--Hausdorff topology,
where $p_{\infty}$ is the limit base point of $p$.

It suffices to show that $\partial_{\T}X$ is totally bounded.
Let $r \in (0,\pi)$,
and
let $Z_r$ be a maximal $r$-separated subset of $\partial_{\T}X$.
Now the sequence of compact metric balls 
$B_1(p; \lambda_iX)$ of radius $1$ around $p$ in $\lambda_iX$ 
converges to the compact ball $B_1(p_{\infty})$ in $C_{\infty}^{(\lambda_i)}X$.
For some sequence $(\epsilon_i) \in O_{\infty}$,
for each $i$
we can take an $\epsilon_i$-approximation
$\varphi_i \colon B_1(p; \lambda_iX) \to B_1(p_{\infty})$.
By the compactness of $B_1(p_{\infty})$,
we find a constant $N_0 \in \N$ such that
every $\sin(r/2)$-separated subset of $B_1(p_{\infty})$
has at most $N_0$ elements.

Suppose that there exists an $r$-separated subset
$\{ \xi_1, \dots, \xi_{N_0+1} \}$ of $Z_r$.
For each $j \in \{ 1, \dots, N_0+1 \}$,
we put $x_j^i := \gamma_{p\xi_j}(1/\lambda_i)$,
where $\gamma_{p\xi_j}$ is the ray in $X$ from $p$ to $\xi_j$.
By selecting a subsequence if necessary,
we may assume that
for each $j \in \{ 1, \dots, N_0+1 \}$
the sequence $\left( \varphi_i(x_j^i) \right)$ converges to
a point $x_j^{\infty}$ in $B_1(p_{\infty})$.
Let $d_X$ be the metric on $X$,
and $d_{\infty}$ the metric on $C_{\infty}^{(\lambda_i)}X$.
From Lemma \ref{lem: scalingcat0}
it follows that
for every $\epsilon \in (0,\infty)$,
for all distinct $j, k \in \{ 1, \dots, N_0+1 \}$ we have
\begin{align*}
d_{\infty} \left( x_j^{\infty}, x_k^{\infty} \right)
&>
d_{\infty} \left( \varphi_i \left( x_j^i \right), 
\varphi_i \left( x_k^i \right) \right) - \epsilon 
> 
\left( \lambda_i d_X \right) \left( x_j^i, x_k^i \right) - \epsilon_i - \epsilon \\
&>
d_{C_0(\partial_{\T}X)}
\left( \xi_j, \xi_k \right) - \epsilon_i - 2 \epsilon
= 
2 \sin \frac{d_{\T} \left( \xi_j, \xi_k \right)}{2} - \epsilon_i - 2\epsilon \\
&\ge
2 \sin \frac{r}{2} - \epsilon_i - 2\epsilon
\ge 
\sin \frac{r}{2} - 2\epsilon
\end{align*}
for all sufficiently large $i$.
Hence the set 
$\left\{ x_1^{\infty}, \dots, x_{N_0+1}^{\infty} \right\}$
is $\sin(r/2)$-separated in $B_1(p_{\infty})$.
This contradicts the choice of the constant $N_0$.

Thus $Z_r$ has at most $N_0$ elements,
and hence it is a finite $(r/2)$-net.
Therefore $\partial_{\T}X$ is totally bounded.
\end{proof}

\subsection{Asymptotic cones and Euclidean cones}

A pointed metric space $(X,p)$
is said to be \emph{isometric to a pointed metric space $(Y,q)$}
if there exists an isometry $f \colon X \to Y$
with $f(p) = q$.

Let $X$ be a proper geodesic metric space.
The 
\emph{Gromov--Hausdorff asymptotic cone}
$C_{\infty}X$ of $X$
is defined by
\[
\left( C_{\infty}X, p_{\infty} \right) := \lim_{\lambda \to 0} (\lambda X, p),
\]
if the limit exists;
more precisely,
if for every sequence $(\lambda_i) \in O_{\infty}$ there exists
the Gromov--Hausdorff asymptotic cone $C_{\infty}^{(\lambda_i)}X$ of $X$ 
with scale $(\lambda_i)$,
and 
if for all $(\lambda_i), (\lambda_i') \in O_{\infty}$ 
the Gromov--Hausdorff asymptotic cones
$\left( C_{\infty}^{(\lambda_i)}X, p_{\infty} \right)$ and
$\left( C_{\infty}^{(\lambda_i')}X, p_{\infty}' \right)$
are isometric to each other.

Next we prove the following:

\begin{lem}\label{lem: approxcone}
Let $X$ be a proper, geodesically complete $\CAT(0)$ space.
Assume that $\partial_{\T}X$ is compact.
Let $p \in X$ be a point.
Then $X$ has the Gromov--Hausdorff asymptotic cone $C_{\infty}X$
such that
$\left( C_{\infty}X, p_{\infty} \right)$ is isometric to 
$\left( C_0(\partial_{\T}X), 0 \right)$,
where $p_{\infty}$ is the limit base point of $p$.
More precisely,
for every $r \in (0,\infty)$,
and for every $\epsilon \in (0,\infty)$,
there exists $\lambda_0 \in (0,\infty)$ such that
for each $\lambda \in (0,\lambda_0)$
the map 
$\Phi_p^{\lambda} \colon B_r(0) \to B_r(p)$
from the ball $B_r(0)$ in $C_0(\partial_{\T}X)$
to the ball $B_r(p)$ in $\lambda X$
defined by 
\[
\Phi_p^{\lambda}(a\xi) := \gamma_{p\xi} \left( a/\lambda \right)
\]
is a surjective $1$-Lipschitz $\epsilon$-approximation.
\end{lem}

\begin{proof}
Take $r \in (0,\infty)$.
Define a function $\theta_r \colon (0,2r) \to (0,\infty)$ by
\[
\theta_r(t) := 2 \sin^{-1} \frac{t}{2r}.
\]
Let $\epsilon \in (0,2r)$.
From Lemma \ref{lem: scalingcat0}
we derive that
if $\xi, \eta \in \partial_{\T}X$ satisfy 
$d_{\T}(\xi,\eta) < \theta_r(\epsilon)$,
then for every $\lambda \in (0,\infty)$ we have
\[
\left( \lambda d_X \right)
\left( \gamma_{p\xi} \left( r/\lambda \right), 
\gamma_{p\eta} \left( r/\lambda \right) \right)
< \epsilon.
\]
Since $\partial_{\T}X$ is compact,
we can find a finite $\theta_r(\epsilon)$-net $\{ \zeta_1, \dots, \zeta_m \}$
of $\partial_{\T}X$.
By Lemma \ref{lem: scalingcat0},
there exists $\lambda_0 \in (0,\infty)$ such that
for every $\lambda \in (0,\lambda_0)$ we have
\[
\left\vert 
d_{C_0(\partial_{\T}X)} 
\left( (l_1\epsilon) \zeta_{m_1}, (l_2\epsilon) \zeta_{m_2} \right)
-
\left( \lambda d_X \right)
\left( \gamma_{p\zeta_{m_1}} \left( l_1 \epsilon/\lambda \right), 
\gamma_{p\zeta_{m_2}} \left( l_2 \epsilon/\lambda \right) \right)
\right\vert
< \epsilon
\]
for all $l_1, l_2 \in \{ 0, 1, \dots, \lfloor r/\epsilon \rfloor \}$
and 
for all $m_1, m_2 \in \{ 1, \dots, m \}$.

Fix $\lambda \in (0,\lambda_0)$.
As shown in Lemma \ref{lem: scalingcat0},
the map $\Phi_p^{\lambda}$ is $1$-Lipschitz.
From the geodesical completeness of $X$
it follows that $\Phi_p^{\lambda}$ is surjective.
To verify that
$\Phi_p^{\lambda}$ is a $9\epsilon$-approximation,
we pick $\xi_1, \xi_2 \in \partial_{\T}X$, and $a_1, a_2 \in (0,r]$.
For some $m_1, m_2 \in \{ 1, \dots, m \}$,
we have $\xi_1 \in U_{\theta_r(\epsilon)}(\zeta_{m_1})$ and
$\xi_2 \in U_{\theta_r(\epsilon)}(\zeta_{m_2})$.
In addition,
for some $l_1, l_2 \in \{ 1, \dots, \lfloor r/\epsilon \rfloor \}$,
we have
$\vert l_1\epsilon - a_1 \vert < \epsilon$ 
and
$\vert l_2\epsilon - a_2 \vert < \epsilon$.
Then
\begin{align*}
d_{C_0(\partial_{\T}X)} \left( a_1\xi_1, a_2\xi_2 \right)
&<
d_{C_0(\partial_{\T}X)} 
\left( (l_1\epsilon) \zeta_{m_1}, (l_2\epsilon) \zeta_{m_2} \right)
+ 4\epsilon \\
&<
\left( \lambda d_X \right)
\left( \gamma_{p\zeta_{m_1}} \left( l_1 \epsilon/\lambda \right), 
\gamma_{p\zeta_{m_2}} \left( l_2 \epsilon/\lambda \right) \right)
+ 5 \epsilon \\
&<
\left( \lambda d_X \right)
\left( \gamma_{p\xi_1} \left( a_1/\lambda \right), 
\gamma_{p\xi_2} \left( a_2/\lambda \right) \right) + 9\epsilon. 
\end{align*}
This implies that $\Phi_p^{\lambda}$ is a $9\epsilon$-approximation.
Thus $X$ has the Gromov--Hausdorff asymptotic cone $C_{\infty}X$
isometric to $C_0(\partial_{\T}X)$.
\end{proof}

In summary,
we conclude the following:

\begin{prop}\label{prop: accat}
Let $X$ be a proper, geodesically complete $\CAT(0)$ space.
Then the following are equivalent:
\begin{enumerate}
\item
$X$ has the Gromov--Hausdorff asymptotic cone $C_{\infty}X$;
\item
$X$ has a Gromov--Hausdorff asymptotic cone $C_{\infty}^{(\lambda_i)}X$
with some scale $(\lambda_i)$;
\item
$\partial_{\T}X$ is compact.
\end{enumerate}
In this case,
for every $p \in X$
the pointed limit $\left( C_{\infty}X, p_{\infty} \right)$ is 
isometric to $\left( C_0(\partial_{\T}X), 0 \right)$,
where $p_{\infty}$ is the limit base point of $p$.
In particular,
the following hold:
\begin{enumerate}
\item
$\partial_{\T}X$ is geodesically complete,
and not a singleton;
moreover, 
if $X$ is doubling,
then $\partial_{\T}X$ is doubling;
\item
if $\dim X = n$, 
then $\dim \partial_{\T}X = n-1$;
\item
if $X$ is pure-dimensional,
then $\partial_{\T}X$ is also pure-dimensional.
\end{enumerate}
\end{prop}

\begin{proof}
In Lemma \ref{lem: conecpt} and Lemma \ref{lem: approxcone},
we already show the implications
(2) $\Rightarrow$ (3)
and 
(3) $\Rightarrow$ (1).
As shown in Lemma \ref{lem: approxcone},
we see that 
$X$ has the Gromov--Hausdorff asymptotic cone $C_{\infty}X$
such that
$\left( C_{\infty}X, p_{\infty} \right)$
is isometric to 
$\left( C_0(\partial_{\T}X), 0 \right)$
for the limit base point $p_{\infty}$.

(1)
Since $C_0(\partial_{\T}X)$ is isometric to $C_{\infty}X$,
it is geodesically complete,
and hence $\partial_{\T}X$ 
is geodesically complete and not a singleton;
moreover,
if $X$ is $N$-doubling,
then so is $C_{\infty}X$.
Hence $\partial_{\T}X$ is doubling.

(2)
Assume that $\dim X = n$.
Applying Lemma \ref{lem: npure} 
to the Gromov--Hausdorff asymptotic cone $C_{\infty}X$,
we see $\dim C_{\infty}X \le n$.
On the other hand,
since there exists a point $x \in X$ with $\dim \Sigma_xX = n-1$,
we have $\dim \Sigma_{p_{\infty}}(C_{\infty}X) \ge n-1$
(\cite[Lemma 11.5]{lytchak-nagano1});
in particular,
we see $\dim C_{\infty}X \ge n$.
Hence $\dim C_0(\partial_{\T}X) = n$,
and $\dim \partial_{\T}X = n-1$.

(3) 
Assume that $X$ is purely $n$-dimensional.
From Lemma \ref{lem: npure} it follows that
$C_{\infty}X$ is purely $n$-dimensional.
Hence $C_0(\partial_{\T}X)$ is purely $n$-dimensional.
Therefore $\partial_{\T}X$ is purely $(n-1)$-dimensional.
\end{proof}

\begin{rem}\label{rem: afteraccat}
Let $X$ be a complete $\CAT(0)$ space.
If $X$ has telescopic dimension $\le n$
in the sense of Caprace--Lytchak \cite{caprace-lytchak},
then the Tits boundary $\partial_{\T}X$ has geometric dimension $\le n-1$
in the sense of Kleiner \cite{kleiner}
(\cite[Proposition 2.1]{caprace-lytchak}).
If $X$ is proper and $\dim X \le n$,
then $\dim_C \partial_{\T}X \le n-1$,
where $\dim_C \partial_{\T}X$ is the supremum of 
the topological dimensions of compact subsets of $\partial_{\T}X$
(\cite[Proposition 1.8]{fujiwara-nagano-shioya}).
\end{rem}

\subsection{Asymptotic cones and volume growths}

We first show the following volume convergence:

\begin{prop}\label{prop: volconvball}
Let $X$ be a proper, geodesically complete $\CAT(0)$ space of $\dim X = n$.
If $X$ has the Gromov--Hausdorff asymptotic cone $C_{\infty}X$,
then for every $p \in X$ we have
\begin{equation}
\lim_{\lambda \to 0}
\frac{\Haus^n \left( U_{1/\lambda}(p) \right)}
{\omega_0^n \left( 1/\lambda \right)}
= \frac{\Haus^{n-1} \left( \partial_{\T}X \right)}
{\Haus^{n-1} \left( \Sph^{n-1} \right)}.
\label{eqn: volconvballi}
\end{equation}
\end{prop}

\begin{proof}
Let $(\lambda_i) \in O_{\infty}$.
By Proposition \ref{prop: accat},
the sequence $(\lambda_i X, p)$ of the pointed proper metric spaces
converges to the pointed proper metric space $(C_0(\partial_{\T}X),0)$ 
in the pointed Gromov--Hausdorff topology. 
For each $i$,
let $U_1(p;\lambda_iX)$ be the open metric ball of radius $1$
around $p$ in $\lambda_i X$.
By Proposition \ref{prop: accat} (2),
for the open ball $U_1(0)$ in $C_0(\partial_{\T}X)$
we see $\dim U_1(0) = n$.
Therefore 
from Theorem \ref{thm: volconv} we derive
\begin{equation}
\lim_{i \to \infty} 
\Haus^n \left( U_1(p; \lambda_i X) \right) = \Haus^n \left( U_1(0) \right).
\label{eqn: volconvball1}
\end{equation}
For the open ball $U_{1/\lambda_i}(p)$ in $X$,
we have
\begin{equation}
\Haus^n \left( U_1(p; \lambda_i X) \right) = 
\lambda_i^n \Haus^n \left( U_{1/\lambda_i}(p) \right)
\label{eqn: volconvball2}
\end{equation}
for all $i$.
Combining \eqref{eqn: volconvball1} and \eqref{eqn: volconvball2}
leads to
\[
\lim_{i \to \infty}
\frac{\Haus^n \left( U_{1/\lambda_i}(p) \right)}
{\omega_0^n \left( 1/\lambda_i \right)}
= \frac{\Haus^n \left( U_1(0) \right)}{\omega_0^n(1)}
= \frac{\Haus^{n-1} \left( \partial_{\T}X \right)}
{\Haus^{n-1} \left( \Sph^{n-1} \right)}.
\]
The last equality follows from
\cite[Lemma 6.4]{nagano2}
and the fact that both $C_0(\partial_{\T}X)$ and $\R^n$
are Euclidean cones.
Thus we obtain the desired equality \eqref{eqn: volconvballi}.
\end{proof}

We next prove the following lemma:

\begin{lem}\label{lem: chaccat0}
For every $n \in \N$,
and for every $c \in [1,\infty)$,
there exists a constant $N_0 \in \N$
satisfying the following property:
Let $X$ be
a purely $n$-dimensional, proper, geodesically complete $\CAT(0)$ space.
If $X$ satisfies $\G_0^n(X) \le c$,
then $X$ is $N_0$-doubling.
\end{lem}

\begin{proof}
It suffices to find a constant $N_0 \in \N$ depending only on $n$ and $c$
such that
for every $r \in (0,\infty)$ and for every $x \in X$,
any $r$-separated set contained in $U_{2r}(x)$
has cardinality $\le N_0$.
Take a maximal $r$-separated set $\{ x_1, \dots, x_N \}$ in $B_{2r}(x)$,
so that $B_{2r}(x)$ is contained in $\bigcup_{j=1}^N U_r(x_j)$.
Observe that $U_{r/2}(x_j) \cap U_{r/2}(x_k)$ is empty
for all distinct $j, k \in \{ 1, \dots, N \}$.
Since $X$ is purely $n$-dimensional,
for each $j \in \{ 1, \dots, N \}$
the value $\Haus^n (U_{r/2}(x_j))$ is positive and finite.
Choose a number $k_0 \in \{ 1, \dots, N \}$ such that
$\Haus^n (U_{r/2}(x_{k_0}))$
is minimal among $\Haus^n (U_{r/2}(x_1)), \dots, \Haus^n (U_{r/2}(x_N))$.
Notice that
for each $j \in \{ 1, \dots, N \}$
the ball $U_{r/2}(x_j)$ is contained in $U_{5r/2}(x_{k_0})$.
Hence we have
\[
N \Haus^n \left( U_{r/2}(x_{k_0}) \right)
\le \Haus^n \left( \bigsqcup_{j=1}^N U_{r/2}(x_j) \right) 
\le \Haus^n \left( U_{5r/2}(x_{k_0}) \right).
\]
From Propositions \ref{prop: abvolcomp} and \ref{prop: relvolcomp},
and from $\G_0^n(X) \le c$,
we derive
\begin{align*}
N &\le
\frac{\Haus^n \left( U_{5r/2}(x_{k_0}) \right)}
{\Haus^n \left( U_{r/2}(x_{k_0}) \right)}
\le
\frac{\Haus^n \left( U_{5r/2}(x_{k_0}) \right)}
{\omega_0^n \left( r/2 \right)}
=
5^n
\frac{\Haus^n \left( U_{5r/2}(x_{k_0}) \right)}
{\omega_0^n \left( 5r/2 \right)} \\
&\le 5^n \G_0^n(X) \le 5^nc.
\end{align*}
Letting $N_0 := \lceil 5^nc \rceil$ completes the proof.
\end{proof}

Summing up,
we conclude the following:

\begin{prop}\label{prop: chaccat}
Let $X$ be a purely $n$-dimensional, proper, 
geodesically complete $\CAT(0)$ space.
Then the following are equivalent:
\begin{enumerate}
\item
$X$ is doubling;
\item
$X$ has the Gromov--Hausdorff asymptotic cone $C_{\infty}X$;
\item
for some $c \in [1,\infty)$
we have
$\G_0^n(X) \le c$.
\end{enumerate}
In this case, 
for every $x \in X$
we have
\begin{equation}
\frac{\Haus^{n-1} \left( \Sigma_xX \right)}
{\Haus^{n-1} \left( \Sph^{n-1} \right)}
\le \frac{\Haus^{n-1} \left( \partial_{\T}X \right)}
{\Haus^{n-1} \left( \Sph^{n-1} \right)}
= \G_0^n(X).
\label{eqn: chaccati}
\end{equation}
\end{prop}

\begin{proof}
In Proposition \ref{prop: volconvball}
and Lemma \ref{lem: chaccat0},
we already verify
(2) $\Rightarrow$ (3)
and 
(3) $\Rightarrow$ (1),
respectively.
The implication
(1) $\Rightarrow$ (2)
follows from Lemma \ref{lem: acdoubling}
and Proposition \ref{prop: accat}.

Now we show \eqref{eqn: chaccati}.
For every $x \in X$,
there exists a $1$-Lipschitz map
$f_x \colon \partial_{\T}X \to \Sigma_xX$ defined by
$f_x(\xi) := \xi_x'$,
where $\xi_x' \in \Sigma_xX$ is the starting direction of $x\xi$ at $x$.
Since $X$ is geodesically complete,
the map $f_x$ is surjective.
Hence we see the inequality in \eqref{eqn: chaccati}.
From Proposition \ref{prop: volconvball} we derive
the equality in \eqref{eqn: chaccati}.
This finishes the proof.
\end{proof}

\begin{rem}\label{rem: gcba-convergence}
Cavallucci and Sambusetti \cite{cavallucci-sambusetti} prove
the compactness results for various classes consisting of 
$\GCBA$ spaces with respect to the Gromov--Hausdorff topology,
and to the pointed one,
which are closely related to the contents in this section.
\end{rem}


\section{Homotopy at infinity}

The goal of this section is to prove that
if a proper, geodesically complete $\CAT(0)$ space
has the Gromov--Hausdorff asymptotic cone,
then any sufficiently large metric sphere is homotopy equivalent to
the Tits ideal boundary.
In the proof,
we use the homotopic stability of fibers of strainer maps
discussed by Lytchak and the author \cite{lytchak-nagano1}.

\subsection{Almost spherical points}

Following \cite[Section 6]{lytchak-nagano1},
we recall the notions of almost spherical points
for $\CAT(1)$ spaces.

Let $\delta \in (0,\infty)$.
Let $\Sigma$ be a compact, geodesically complete $\CAT(1)$ space 
with metric $d_{\Sigma}$
of diameter $\pi$.
A point $\xi \in \Sigma$ is said to be
\emph{$\delta$-spherical}
if there exists $\eta \in \Sigma$
such that for every $\zeta \in \Sigma$ we have
\[
d_{\Sigma}(\xi,\zeta) + d_{\Sigma}(\zeta,\eta) < \pi + \delta;
\]
in this case, 
the pair of $\xi$ and $\eta$ are said to be \emph{opposite}.
For a point $\xi \in \Sigma$,
a point $\eta \in \Sigma$ is an
\emph{antipode of $\xi$}
if $d_{\Sigma}(\xi,\eta) = \pi$,
and the set of all antipodes of $\xi$ is denote by $\Ant(\xi)$.

From the triangle inequality and the extendability of geodesics
to length $\pi$,
we derive the following (\cite[Lemma 6.3]{lytchak-nagano1}):

\begin{lem}\label{lem: opposite}
\emph{(\cite{lytchak-nagano1})}
Let $\Sigma$ be a compact geodesically complete $\CAT(1)$ space 
with metric $d_{\Sigma}$
of diameter $\pi$.
Then $\xi, \eta \in \Sigma$ are opposite $\delta$-spherical points
if and only if $d_{\Sigma}(\eta,\zeta) < \delta$ for any $\zeta \in \Ant(\xi)$;
in this case, $d_{\Sigma}(\xi,\eta) > \pi - \delta$,
and $\Ant (\xi)$ has diameter $< 2\delta$;
moreover, for every $\zeta \in \Ant(\xi)$ the pair of $\xi$ and $\zeta$
are opposite $2\delta$-spherical points.
\end{lem}

An $m$-tuple $(\xi_1, \dots, \xi_m)$ of points in $\Sigma$ is said to be
\emph{$\delta$-spherical}
if there exists another 
$m$-tuple $(\eta_1, \dots, \eta_m)$ of points in $\Sigma$
such that
\begin{enumerate}
\item
$\xi_j$ and $\eta_j$ are opposite $\delta$-spherical points
for all $j \in \{ 1, \dots, m \}$;
\item
$d_{\Sigma}(\xi_j,\xi_k)$, $d_{\Sigma}(\eta_j,\eta_k)$, and 
$d_{\Sigma}(\xi_j,\eta_k)$
are smaller than $\pi/2 + \delta$
for all distinct $j, k \in \{ 1, \dots, m \}$;
\end{enumerate}
in this case, 
the pair of $(\xi_1, \dots, \xi_m)$ and $(\eta_1, \dots, \eta_m)$
are \emph{opposite}.

By Lemma \ref{lem: opposite},
we have the following (\cite[Corollary 6.5]{lytchak-nagano1}):

\begin{lem}\label{lem: multi-sph}
\emph{(\cite{lytchak-nagano1})}
Let $\Sigma$ be a compact geodesically complete $\CAT(1)$ space
with metric $d_{\Sigma}$
of diameter $\pi$,
and 
let $(\xi_1, \dots, \xi_m)$ be an $m$-tuple of $\delta$-spherical points 
in $\Sigma$.
Then the following hold:
\begin{enumerate}
\item
if $(\xi_1, \dots, \xi_m)$ is a $\delta$-spherical $m$-tuple,
then 
\[
\frac{\pi}{2} - 2\delta < d_{\Sigma}(\xi_j,\xi_k) < \frac{\pi}{2} + \delta
\]
for all distinct $j, k \in \{ 1, \dots, m \}$;
\item
if for all distinct $j, k \in \{ 1, \dots, m \}$ we have
\[
\frac{\pi}{2} - \delta < d_{\Sigma}(\xi_j,\xi_k) < \frac{\pi}{2} + \delta,
\]
then for each $\eta_j \in \Sigma$ with $d_{\Sigma}(\xi_j,\eta_j) = \pi$,
$j \in \{ 1, \dots, m \}$,
the $m$-tuples $(\xi_1, \dots, \xi_m)$ and $(\eta_1, \dots, \eta_m)$
are opposite $2\delta$-spherical. 
\end{enumerate}
\end{lem}

\subsection{Strainers}

Based on \cite[Section 7]{lytchak-nagano1},
we discuss the notions of strainers 
for proper geodesically complete $\CAT(\kappa)$ spaces.

Let $X$ be a proper, geodesically complete $\CAT(\kappa)$ space.
A point $x \in X$ is said to be
\emph{$(m,\delta)$-strained}
if $\Sigma_xX$ admits a $\delta$-spherical $m$-tuple.

We say that an open metric ball in $X$ is 
\emph{tiny}
if the radius is smaller than $\min \{ 1, D_{\kappa}/100 \}$.
An open metric ball $U_{r_0}(x_0)$ in $X$ 
\emph{has capacity bounded by $N$}
if $B_{5r_0}(x_0)$ is $N$-doubling.
Notice that
every tiny ball has capacity bounded by $N$
for some $N$
(\cite[Proposition 5.1]{lytchak-nagano1}).

Let $U_{r_0}(x_0)$ be a tiny ball in $X$.
For a point $x \in U_{r_0}(x_0)$,
we say that
an $m$-tuple $(p_1, \dots, p_m)$ of points in $B_{5r_0}(x_0) - \{ x \}$ is an
\emph{$(m,\delta)$-strainer at $x$} 
if the $m$-tuple $((p_1)_x', \dots, (p_m)_x')$ 
of the starting directions at $x$
is $\delta$-spherical in $\Sigma_xX$.
For a subset $W$ of $U_{r_0}(x_0)$,
we say that
an $m$-tuple $(p_1, \dots, p_m)$ of points in $B_{5r_0}(x_0) - W$ is an
\emph{$(m,\delta)$-strainer at $W$} 
if for every $x \in W$
the $m$-tuple $(p_1, \dots, p_m)$
is an $(m,\delta)$-strainer at $x$.

We review the following basic observation
(\cite[Proposition 7.3]{lytchak-nagano1}):

\begin{prop}\label{prop: 1dstr}
\emph{(\cite{lytchak-nagano1})}
Let $\delta \in (0,\infty)$.
Let $X$ be a proper, geodesically complete $\CAT(\kappa)$ space.
Then for every $p \in X$ there exists $r \in (0,D_{\kappa})$
such that the point $p$ is a $(1,\delta)$-strainer at $U_r(p) - \{ p \}$.
\end{prop}

We know the following relation with dimension
(\cite[Lemma 11.7]{lytchak-nagano1}):

\begin{lem}\label{lem: strdim}
\emph{(\cite{lytchak-nagano1})}
Let $\delta \in (0,\infty)$.
Let $X$ be a proper, geodesically complete $\CAT(\kappa)$ space.
If $4m\delta < 1$,
then for every $(m,\delta)$-strained point $x \in X$ we have
$\dim T_xX \ge m$.
\end{lem}

Notice that
if a tiny ball $U_{r_0}(x_0)$ in $X$ satisfies $\dim U_{r_0}(x_0) = n$,
then $n$ is the largest number such that
there exists a point $x \in U_{r_0}(x_0)$
at which $\Sigma_xX$ is isometric to $\R^n$.
In addition,
$n$ is the largest number such that
there exists an $(n,1/4n)$-strained point in $U_{r_0}(x_0)$
(see \cite[Proposition 11.1]{lytchak-nagano1}).

\subsection{Strainer maps}

As well as \cite[Section 8]{lytchak-nagano1},
we discuss the notions of strainer maps 
for proper geodesically complete $\CAT(\kappa)$ spaces.

Let $X$ be a proper, geodesically complete $\CAT(\kappa)$ space
with metric $d_X$.
For a point $p \in X$, we denote by $d_p$
the distance function from $p$ defined by $d_p(x) := d_X(p,x)$.
We say that a map $f \colon U \to \R^m$
from an open subset $U$ in a tiny ball in $X$ is an 
\emph{$(m,\delta)$-strainer map}
if there exists an $(m,\delta)$-strainer $(p_1, \dots, p_m)$ at $U$
with $f = (d_{p_1}, \dots, d_{p_m})$.
Every $(m,\delta)$-strainer map is $2\sqrt{m}$-Lipschitz.
If $4m\delta \le 1$,
then every $(m,\delta)$-strainer map $f \colon U \to \R^m$
is $2\sqrt{m}$-Lipschitz
and $2\sqrt{m}$-open;
in particular, 
the Hausdorff dimension of $U$ is at least $m$ 
(\cite[Lemma 8.2]{lytchak-nagano1}).
Moreover,
the Lipschitz constant and the openness constant
of strainer maps can be chosen close to $1$
in the following sense:
For a given constant $N \in \N$, 
for every $\epsilon \in (0,1)$,
and for every $m \in \N$,
there exists $\delta \in (0,\infty)$ such that
every $(m,\delta)$-strainer map with domain
contained in a tiny ball of capacity bounded by $N$
is $(1+\epsilon)$-Lipschitz and $(1+\epsilon)$-open
(\cite[Corollary 8.4]{lytchak-nagano1}).

\subsection{Homotopic stability of fibers of strainer maps}

We now recall the following homotopic stability theorem of 
\cite{lytchak-nagano1} 
concerning fibers of strainer maps 
on proper geodesically complete $\CAT(\kappa)$ spaces.
(see a more general statement \cite[Theorem 13.1]{lytchak-nagano1}):

\begin{thm}\label{thm: hstab}
\emph{(\cite{lytchak-nagano1})}
Let $(X_i)$ be a sequence of proper geodesically complete $\CAT(\kappa)$ spaces,
and let $X$ be a proper, geodesically complete $\CAT(\kappa)$ space.
Assume that 
a sequence $(U_{r_0}(x_i))$ of tiny balls 
with radius $r_0$ in $X_k$
has capacity uniformly bounded by $N_0$,
and a sequence $(B_{10r_0}(x_i),x_i)$ of pointed compact metric balls in $X_i$
converges to a pointed compact metric ball $(B_{10r_0}(x),x)$ in $X$
in the pointed Gromov--Hausdorff topology.
For an open subset $U$ contained in $U_{r_0}(x)$,
and for an $(m,\delta)$-strainer $(p_1,\dots,p_m)$ at $U$
with $20 m \delta \le 1$,
let $f \colon U \to \R^m$ be an $(m,\delta)$-strainer map with
$f = (d_{p_1},\dots,d_{p_m})$.
Assume that for some $c \in \R^m$ the fiber $f^{-1}(\{c\})$ is compact.
Let $(\Pi_i)$ be a sequence of compact subsets contained in $U_{r_0}(x_i)$
such that $(\Pi_i)$ converges to $f^{-1}(\{c\})$ 
in the Gromov--Hausdorff topology.
Let $(p_{1,i}, \dots, p_{m,i})$ be a sequence 
of $m$-tuples of points in $B_{5r_0}(x_i)$ such that
$(p_{j,i})$ converges to $p_j$ for all $j \in \{ 1, \dots, m \}$.
Then for every sequence $(c_i)$ in $\R^n$ with $\lim_{i \to \infty} c_i = c$
there exists a positive number $r \in (0,\infty)$
such that for the sequence $(f_i)$ of the maps
$f_i \colon U_r(\Pi_i) \to \R^m$ given by
$f_i = (d_{p_{1,i}}, \dots, d_{p_{m,i}})$
the following hold for all sufficiently large $i$:
\begin{enumerate}
\item
each $f_i$ is an $(m,\delta)$-strainer map;
\item
the fiber $f_i^{-1}(\{c_i\})$ is compact,
and the sequence $(f_i^{-1}(\{c_i\}))$ converges to $f^{-1}(\{c\})$
in the Gromov--Hausdorff topology.
\item
the fiber $f_i^{-1}(\{c_i\})$ is homotopy equivalent to the fiber $f^{-1}(\{c\})$.
\end{enumerate}
\end{thm}

\subsection{Homotopic stability at infinity}

We deduce the following:

\begin{thm}\label{thm: hstabinf}
Let $X$ be a proper, geodesically complete $\CAT(0)$ space.
If $X$ has the Gromov--Hausdorff asymptotic cone $C_{\infty}X$,
then for every $p \in X$ there exists a sufficiently large $t_0 \in (0,\infty)$
such that for an arbitrary $t \in [t_0,\infty)$ the metric sphere
$S_t(x)$ is homotopy equivalent to 
the Tits boundary $\partial_{\T}X$.
\end{thm}

\begin{proof}
Let $p \in X$ be arbitrary.
Take a sequence $(\lambda_i)$ in $(0,\infty)$ with 
$\lim_{i \to \infty} \lambda_i = 0$.
Since $X$ has the Gromov--Hausdorff asymptotic cone $C_{\infty}X$,
as seen in Proposition \ref{prop: accat},
the pointed Gromov--Hausdorff limit 
$\lim_{i \to \infty} (\lambda_i X, p)$ exists,
and it is isometric to $(C_0(\partial_{\T}X),0)$;
moreover,
$C_0(\partial_{\T}X)$ is proper and geodesically complete.

Take a unit open metric ball $U_1(0)$ in $C_0(\partial_{\T}X)$
as a tiny ball.
Since the sequence $(\lambda_iX,p)$ converges to 
$(C_0(\partial_{\T}X),0)$,
the sequence $(U_1(p))$ of tiny balls in $\lambda_i X$
has capacity uniformly bounded.
Furthermore, 
the sequence $(B_{10}(p))$ of 
pointed compact metric balls in $\lambda_i X$
converges to the pointed compact metric ball $(B_{10}(0),0)$ in $C_0(\partial_{\T}X)$.

Choose $\delta \in (0,1)$ with $20\delta \le 1$.
Let $f \colon U_1(0) - \{0\} \to \R$ denote 
the distance function from $0$ defined by
$f(t \xi) := t$.
Since $C_0(\partial_{\T}X)$ has the structure of a Euclidean cone,
we see that $f$ is a $(1,\delta)$-strainer map.
Take $c \in (0,1)$.
The fiber $f^{-1}(\{c\})$ is the compact metric sphere $S_c(0)$
in $C_0(\partial_{\T}X)$,
and it is homeomorphic to $\partial_{\T}X$.

For each $i$,
let $f_i \colon U_1(p) - \{p\} \to \R$ be
the distance function from $p$ defined by
$f_i(x) := d_{\lambda_i X}(p,x)$,
where $U_1(p)$ is the open metric ball in $\lambda_iX$
and $d_{\lambda_iX}$ is the metric on $\lambda_iX$.
Then the fiber $f_i^{-1}(\{c\})$ is the compact metric sphere $S_c(p)$
in $\lambda_i X$.
Observe that the sequence $(f_i^{-1}(\{c\}))$ converges to $f^{-1}(\{c\})$
in the Gromov--Hausdorff topology.
From Theorem \ref{thm: hstab} it follows that
each fiber $f_i^{-1}(\{c\})$ is homotopy equivalent to the fiber $f^{-1}(\{c\})$
for all sufficiently large $i$.
Thus the metric sphere $S_{c/\lambda_i}(p)$ in $X$
is homotopy equivalent to $\partial_{\T}X$
for all sufficiently large $i$.
This finishes the proof.
\end{proof}

\subsection{Simply connectedness at infinity}

A topological $n$-manifold $M$ is said to be
\emph{simply connected at infinity}
if there exists a sequence $(K_i)$ of compact subsets of $M$
with $M = \bigcup_{i=1}^{\infty} K_i$ satisfying the following properties
for all $i$:
(1) 
$K_i \subset K_{i+1}$;
(2)
every loop in $M-K_{i+1}$ is contractible in $M-K_i$.
A contractible topological $n$-manifold
is homeomorphic to $\R^n$
if and only if it is simply connected at infinity
(\cite{brown, rolfsen} for the case of $n = 3$, 
\cite{freedman} for $n = 4$, 
and \cite{stallings} for $n \ge 5$).

From Theorem \ref{thm: hstabinf}
we derive the following:

\begin{thm}\label{thm: sci}
Let $X$ be a proper, geodesically complete $\CAT(0)$ space.
Assume that $X$ has the Gromov--Hausdorff asymptotic cone $C_{\infty}X$.
If $X$ is a topological $n$-manifold,
and if $\partial_{\T}X$ is simply connected,
then $X$ is homeomorphic to $\R^n$.
\end{thm}

\begin{proof}
We prove that $X$ is simply connected at infinity.
Let $p \in X$ be arbitrary.
Let $t_0 \in (0,\infty)$ be sufficiently large as in Theorem \ref{thm: hstabinf},
so that for every $t \in [t_0,\infty)$
the metric sphere $S_t(p)$ is homotopy equivalent to $\partial_{\T}X$.
Choose a monotone decreasing sequence $(\lambda_i)$ in $(0,1/t_0)$
with $\lim_{i \to \infty} \lambda_i = 0$.
Then the sequence
$(B_{1/\lambda_i}(p))$ of compact subsets of $X$
satisfies $X = \bigcup_{i=1}^{\infty} B_{1/\lambda_i}(p)$
and $B_{1/\lambda_i}(p) \subset B_{1/\lambda_{i+1}}(p)$ for each $i$.

Now it suffices to show that 
every loop in $X-B_{1/\lambda_{i+1}}(p)$ is contractible in $X-B_{1/\lambda_i}(p)$.
Let $\sigma \colon \Sph^1 \to X-B_{1/\lambda_{i+1}}(p)$ be an arbitrary loop.
Set $t_i := 1/\lambda_{i+1}$.
Let
$\varphi_i \colon X-B_{t_i}(p) \to S_{t_i}(p)$
be the geodesic contraction map defined by
$\varphi_i(x) := \gamma_{px}(t_i)$,
where $\gamma_{px} \colon [0,d_X(p,x)] \to X$ is the 
unit-speed geodesic in $X$ from $p$ to $x$,
and $d_X$ is the metric on $X$.
Since $X$ is $\CAT(0)$,
the map $\varphi_i$ is continuous.
Then the loops $\sigma$ and $\varphi_i \circ \sigma$
can be joined by a homotopy along
the geodesics emanating from $p$;
indeed,
the map
$h_i \colon \Sph^1 \times [0,1] \to X-B_{1/\lambda_i}(p)$ defined by
\[
h_i(s,t) := \gamma_{p\sigma(s)} \left( (1-t) d_X(p,\sigma(s)) + tt_i \right)
\]
is a homotopy in $X-B_{1/\lambda_i}(p)$
from $\sigma$ to $\varphi_i \circ \sigma$.
Since $\partial_{\T}X$ is simply connected,
so is $S_{t_i}(p)$.
Hence $\varphi_i \circ \sigma$ is contractible in $S_{t_i}(p)$,
and hence $\sigma$ is contractible in $X-B_{1/\lambda_i}(p)$.
This completes the proof.
\end{proof}


\section{Asymptotic geometric regularity}

\subsection{Busemann functions on CAT(0) spaces}

Let $X$ be a complete $\CAT(0)$ space with metric $d_X$.
Let $\gamma \colon [0,\infty) \to X$ be a ray in $X$.
The \emph{Busemann function $b_{\gamma} \colon X \to \R$} along $\gamma$
is defined by
\[
b_{\gamma}(x) := 
\lim_{t \to \infty} \left( d_X \left( x, \gamma(t) \right) - t \right).
\]
The Busemann function $b_{\gamma}$ is $1$-Lipschitz and convex.
For $r \in \R$, 
we denote by $B_r(\gamma)$
the closed $(-r)$-horoball $b_{\gamma}^{-1} ((-\infty,-r])$.
Note that
for every $x \in X$,
and for every $r \in (0,\infty)$ with $b_{\gamma}(x) > -r$,
we have
\begin{equation}
d_X \left( x, B_r(\gamma) \right) = b_{\gamma}(x) + r.
\label{eqn: horoball}
\end{equation}
For every $x \in X$, and for every $y \in X - \{ x \}$,
by letting $\xi := \gamma(\infty)$,
we have the first variation formula 
\begin{equation}
(b_{\gamma} \circ \gamma_{xy})_+'(a) = - \cos \angle_x (\xi,y),
\label{eqn: 1vfb}
\end{equation}
where
$\gamma_{xy} \colon [a,b] \to X$ is the geodesic 
from $x$ to $y$,
and $(b_{\gamma} \circ \gamma_{xy})_+'(a)$
is the right derivative of $b_{\gamma} \circ \gamma_{xy}$ at $a$
(see e.g., \cite[Lemma 3.3]{fujiwara-nagano-shioya}).

We first show the following basic property:

\begin{lem}\label{lem: busemann-level}
Let $X$ be a complete $\CAT(0)$ space.
Let $\gamma \colon [0,\infty) \to X$ be a ray in $X$.
Let $\xi := \gamma(\infty)$.
If for distinct $x, y \in X$ we have
$b_{\gamma}(x) = b_{\gamma}(y)$,
then
$\angle_x (\xi,y) \le \pi/2$ and
$\angle_y (\xi,x) \le \pi/2$.
\end{lem}

\begin{proof}
Take distinct $x, y \in X$ with
$b_{\gamma}(x) = b_{\gamma}(y)$.
Let
$\gamma_{xy} \colon [a,b] \to X$ be the geodesic 
from $x$ to $y$.
Since $b_{\gamma}$ is convex along $\gamma_{xy}$,
for every $t \in [a,b]$ we have
$b_{\gamma}(\gamma_{xy}(t)) \le b_{\gamma}(x)$.
From the first variation formula \eqref{eqn: 1vfb}
we derive 
\[
- \cos \angle_x (\xi,y) =
(b_{\gamma} \circ \gamma_{xy})_+'(a) \le 0,
\]
and hence $\angle_x (\xi,y) \le \pi/2$.
Similarly, we see $\angle_y (\xi,x) \le \pi/2$.
\end{proof}

\subsection{Strainers at infinity}

We now introduce the following:

\begin{defn}\label{defn: strinf}
Let $\delta \in (0,\infty)$.
Let $X$ be a proper, geodesically complete $\CAT(0)$ space.
We say that
a point $\xi \in \partial_{\infty}X$ is 
\emph{$\delta$-spherical}
if there exists some $\eta \in \partial_{\infty}X$
such that for every $\zeta \in \partial_{\infty}X$ we have
\[
\angle (\xi,\zeta) + \angle (\zeta,\eta) < \pi + \delta;
\]
in this case, 
the pair of $\xi$ and $\eta$ are said to be \emph{opposite}.
We say that
an $m$-tuple $(\xi_1, \dots, \xi_m)$ of points in $\partial_{\infty}X$ is an
\emph{$(m,\delta)$-strainer at infinity}
if there exists another 
$m$-tuple $(\eta_1, \dots, \eta_m)$ of points in $\partial_{\infty}X$
such that
\begin{enumerate}
\item
$\xi_j$ and $\eta_j$ are opposite $\delta$-spherical points
for all $j \in \{ 1, \dots, m \}$;
\item
$\angle (\xi_j,\xi_k)$, $\angle (\eta_j,\eta_k)$, and $\angle (\xi_j,\eta_k)$
are smaller than $\pi/2 + \delta$
for all distinct $j, k \in \{ 1, \dots, m \}$;
\end{enumerate}
in this case, 
the pair of $(\xi_1, \dots, \xi_m)$ and $(\eta_1, \dots, \eta_m)$
are \emph{opposite}.
\end{defn}

By the definition of the angle metric $\angle$,
we have:

\begin{lem}\label{lem: sphstrinf}
Let $X$ be a proper, geodesically complete $\CAT(0)$ space,
and let $(\xi_1, \dots, \xi_m)$ be an $(m,\delta)$-strainer at infinity.
Then for every $x \in X$ the $m$-tuple
$((\xi_1)_x', \dots, (\xi_m)_x')$ of the directions in $\Sigma_xX$
forms a $\delta$-spherical $m$-tuple.
Moreover,
if $(\xi_1, \dots, \xi_m)$ and $(\eta_1, \dots, \eta_m)$ 
are opposite $(m,\delta)$-strainers at infinity,
then the $m$-tuples
$\left( (\xi_1)_x', \dots, (\xi_m)_x' \right)$ and 
$\left( (\eta_1)_x', \dots, (\eta_m)_x' \right)$
are opposite $\delta$-spherical points in $\Sigma_xX$.
\end{lem}

If there exists an $(m,\delta)$-strainer at infinity, 
then via Lemma \ref{lem: sphstrinf} 
we can apply the local statements obtained in \cite{lytchak-nagano1}.

Similarly to \cite[Lemma 7.6]{lytchak-nagano1},
we verify that
the existence of strainers at infinity guarantees the 
existence of almost flat ideal triangles:

\begin{lem}\label{lem: almflattri}
Let $X$ be a proper, geodesically complete $\CAT(0)$ space.
Let $\xi \in \partial_{\infty}X$ be a $(1,\delta)$-strainer at infinity. 
Then for every pair of distinct points $x, y \in X$
the following hold:
\begin{enumerate}
\item
$\pi - 2\delta < \angle_x (\xi,y) + \angle_y (\xi,x) \le \pi$;
\item
if $b_{\gamma}(x) = b_{\gamma}(y)$ for a ray 
$\gamma$ in $X$ with $\xi = \gamma(\infty)$,
then 
\[
\frac{\pi}{2} - 2\delta < \angle_x (\xi,y) \le \frac{\pi}{2}.
\]
\end{enumerate}
\end{lem}

\begin{proof}
Since $X$ is $\CAT(0)$,
we know $\angle_x (\xi,y) + \angle_y (\xi,x) \le \pi$
(see e.g., \cite[Proposition II.9.3]{bridson-haefliger}).
By Lemma \ref{lem: busemann-level},
if $b_{\gamma}(x) = b_{\gamma}(y)$ for a ray 
$\gamma$ in $X$ with $\xi = \gamma(\infty)$,
then $\angle_x (\xi,y) \le \pi/2$ and
$\angle_y (\xi,x) \le \pi/2$.

Take a point $\eta \in \partial_{\infty}X$ for which $\xi$ and $\eta$
are opposite $(1,\delta)$-strainers at infinity.
Similarly to the case of $\xi$,
we know $\angle_x (\eta,y) + \angle_y (\eta,x) \le \pi$.
By Lemmas \ref{lem: opposite} and \ref{lem: sphstrinf},
we have
\begin{align*}
\angle_x (\xi,y) &+ \angle_x (\eta,y) \ge \angle_x (\xi,\eta) > \pi - \delta, \\
\angle_y (\xi,x) &+ \angle_y (\eta,x) \ge \angle_y (\xi,\eta) > \pi - \delta.
\end{align*}
Therefore $\angle_x (\xi,y) + \angle_y (\xi,x) > \pi - 2\delta$.
Moreover,
if $b_{\gamma}(x) = b_{\gamma}(y)$,
then we obtain $\angle_x (\xi,y) > \pi/2 - 2\delta$.
This proves the lemma.
\end{proof}

\subsection{Strainer maps at infinity}

Let $X$ be a proper, geodesically complete $\CAT(0)$ space.
For an arbitrary $m$-tuple $(\gamma_1, \dots, \gamma_m)$
of rays in $X$,
by letting $f_j := b_{\gamma_j}$ for $j \in \{ 1, \dots, m \}$,
we obtain the map
$F = (f_1, \dots, f_m) \colon X \to \R^m$
with Busemann function coordinates.
Put $\xi_j := \gamma_j(\infty)$ for $j \in \{ 1, \dots, m \}$.
By the first variation formula \eqref{eqn: 1vfb}
for Busemann functions,
the map $F$ is differentiable at all $x \in X$
with differential $D_xF \colon T_xX \to \R^m$ determined as
\[
(D_xF) (rv)
= -r \left( \cos \angle_x \left( (\xi_1)_x', v \right), \dots, 
\cos \angle_x \left( (\xi_m)_x', v \right) \right).
\]

Similarly to \cite[Lemma 8.1]{lytchak-nagano1},
we see the following:

\begin{lem}\label{lem: dsmi}
Set $c := 1/4m$.
Let $X$ be a proper, geodesically complete $\CAT(0)$ space.
Let $(\gamma_1, \dots, \gamma_m)$ be an $m$-tuple of rays in $X$,
and put $f_j := b_{\gamma_j}$ for $j \in \{ 1, \dots, m \}$.
Assume that for every $x \in X$ there exist $m$-tuples
$\left( v_1^{\pm}, \dots, v_m^{\pm} \right)$
of directions in $\Sigma_xX$ satisfying the following:
\begin{enumerate}
\item
$\pm (D_xf_j) \left( v_j^{\pm} \right) > 1 - c$ for all $j \in \{ 1, \dots, m \}$;
\item
$\left\vert (D_xf_k) \left( v_j^{\pm} \right) \right\vert < 2c$
for all distinct $j, k \in \{ 1, \dots, m \}$.
\end{enumerate}
Then the map $F = (f_1, \dots, f_m) \colon X \to \R_1^m$
with Busemann function coordinates
is $2$-open,
where
$\R_1^m$ denotes the $m$-dimensional real vector space 
with $\ell_1$-norm.
\end{lem}

\begin{defn}\label{defn: bsmi}
Let $X$ be a proper, geodesically complete $\CAT(0)$ space.
For an $m$-tuple $(\gamma_1, \dots, \gamma_m)$ of rays in $X$,
let $f_j := b_{\gamma_j}$ 
for the Busemann function along $\gamma_j$ for $j \in \{ 1, \dots, m \}$.
We say that the map $F = (f_1, \dots, f_m) \colon X \to \R^m$ is a 
\emph{Busemann $(m,\delta)$-strainer map} 
if the $m$-tuple $\left( \gamma_1(\infty), \dots, \gamma_m(\infty) \right)$ 
in $\partial_{\infty}X$ is an $(m,\delta)$-stainer at infinity.
\end{defn}

By Lemmas \ref{lem: sphstrinf} and \ref{lem: dsmi},
and by applying the same idea as \cite[Lemma 8.2]{lytchak-nagano1}
to our setting,
we have:

\begin{lem}\label{lem: bsm}
Let $X$ be a proper, geodesically complete $\CAT(0)$ space.
If $4m\delta < 1$,
then every Busemann $(m,\delta)$-strainer map $F \colon X \to \R^m$
is $2\sqrt{m}$-Lipschitz and $2\sqrt{m}$-open;
in particular,
if $X$ admits a Busemann $(m,\delta)$-strainer map, 
then the Hausdorff dimension of $X$ is at least $m$.
\end{lem}

Mimicking the limiting arguments in \cite[Subsection 8.3]{lytchak-nagano1}
together with Lemmas \ref{lem: sphstrinf} and \ref{lem: bsm},
similarly to \cite[Lemma 8.3]{lytchak-nagano1},
we can obtain: 

\begin{lem}\label{lem: ibsm}
For a given constant $N \in \N$,
for every $\epsilon \in (0,1)$, and
for every $m \in \N$,
there exists $\delta \in (0,\infty)$ satisfying the following: 
Let $X$ be a proper, geodesically complete $\CAT(0)$ space,
and let $F \colon X \to \R^m$ be a Busemann $(m,\delta)$-strainer map.
If an open convex subset $U$ of $X$ is $N$-doubling,
then for every $x \in U$
the differential $D_xF \colon T_xX \to \R^m$ satisfies the following:
\begin{enumerate}
\item
$\left\vert (D_xF)(v) \right\vert < 1 + \epsilon$ for all $v \in \Sigma_xX$; 
\item
for every $u \in \Sph^{m-1}$ in $\R^m$
there exists an element $rv \in T_xX$ with $(D_xF)(rv) = u$ and
$r < 1+\epsilon$.
\end{enumerate}
\end{lem}

Due to the Lytchak open map theorem 
\cite[Theorem 1.2]{lytchak2} for metric spaces,
we conclude the following
as well as \cite[Corollary 8.4]{lytchak-nagano1}:

\begin{prop}\label{prop: submbsm}
For a given constant $N \in \N$,
for every $\epsilon \in (0,1)$, and
for every $m \in \N$,
there exists $\delta \in (0,\infty)$ satisfying the following property: 
Let $X$ be a proper, geodesically complete $\CAT(0)$ space,
and let $F \colon X \to \R^m$
be a Busemann $(m,\delta)$-strainer map.
If an open convex subset $U$ of $X$ is $N$-doubling,
then the restriction $F|U$ of $F$ to $U$
is $(1+\epsilon)$-Lipschitz and $(1+\epsilon)$-open.
In particular,
if $X$ is $N$-doubling,
then $F$ is $(1+\epsilon)$-Lipschitz and $(1+\epsilon)$-open.
\end{prop}

\subsection{Differentials of strainer maps at infinity}

Similarly to \cite[Proposition 8.5]{lytchak-nagano1},
we see the following:

\begin{prop}\label{prop: diffbsm}
Let $X$ be a proper, geodesically complete $\CAT(0)$ space,
and let $F \colon X \to \R^m$ be a Busemann $(m,\delta)$-strainer map 
with $4m\delta < 1$.
Let $\gamma \colon [a,b] \to X$ be a geodesic in $X$.
Then for all $s, t \in [a,b)$ 
\[
\left\vert \left( F \circ \gamma \right)_+'(s) - 
\left( F \circ \gamma \right)_+'(t) \right\vert
< 4\delta \sqrt{m},
\]
where $\left( F \circ \gamma \right)_+'(t)$ is the right derivative 
of $F \circ \gamma$ at $t$.
If in addition $\gamma$ contains at least two distinct points
on a single fiber of $F$,
then for all $t \in [a,b)$ we have
\[
\left\vert \left( F \circ \gamma \right)_+'(t) \right\vert
< 6\delta \sqrt{m}.
\]
\end{prop}

\begin{proof}
Let $(\gamma_1, \dots, \gamma_m)$
be the $m$-tuples of the rays in $X$
such that for each $j \in \{ 1, \dots, m \}$
the Busemann function $b_{\gamma_j}$ is the $j$-th coordinate of $F$.
Let $\xi_j := \gamma_j(\infty)$.
For $t \in [a,b)$,
set $\alpha_j(t) := \angle_{\gamma(t)} (\xi_j,\gamma(b))$,
and $\beta_j := \angle_{\gamma(b)} (\xi_j,\gamma(a))$.
By Lemma \ref{lem: almflattri} (1) and \eqref{eqn: 1vfb},
for all $s, t \in [a,b)$ 
\begin{multline*}
\left\vert \left( b_{\gamma_j} \circ \gamma \right)_+'(s) - 
\left( b_{\gamma_j} \circ \gamma \right)_+'(t) \right\vert
= \left\vert - \cos \alpha_j(s) + \cos \alpha_j(t) \right\vert \\
\le \left\vert \cos \alpha_j(s) + \cos \beta_j \right\vert
+ \left\vert - \cos \beta_j + \cos \alpha_j(t) \right\vert
< 4\delta.
\end{multline*}
This implies the first inequality.

To show the second inequality,
we assume that $F(\gamma(r)) = F(\gamma(s))$ 
holds for $r, s \in [a,b]$ with $r < s$.
By Lemma \ref{lem: almflattri} (2) and \eqref{eqn: 1vfb},
\[
\left\vert \left( b_{\gamma_j} \circ \gamma \right)_+'(r) \right\vert
= \left\vert \cos \alpha_j(r) \right\vert < 2\delta.
\]
In particular, 
we have
$\left\vert \left( F \circ \gamma \right)_+'(r) \right\vert
< 2\delta \sqrt{m}$.
This together with the first inequality leads to
the second one for all $t \in [a,b)$.
\end{proof}

\subsection{Fully strained CAT(0) spaces}

We recall that
a $c$-open map from a complete metric space is surjective;
moreover, 
a $c$-Lipschitz map from a complete metric space
for some $c \in [1,\infty)$
is a $c$-bi-Lipschitz homeomorphism
if and only if it is an injective $c$-open map
(as mentioned in Subsection 2.2).

Now we prove the following:

\begin{prop}\label{prop: bilipbsm0}
Let $X$ be a proper, geodesically complete $\CAT(0)$ space 
of $\dim X = n$,
and let $F \colon X \to \R^n$ be a Busemann $(n,\delta)$-strainer map 
with $100n\delta < 1$.
Then $F$ is a $2\sqrt{n}$-bi-Lipschitz homeomorphism
\end{prop}

\begin{proof}
Let $(\gamma_1, \dots, \gamma_n)$
be the $n$-tuples of the rays in $X$
such that for each $j \in \{ 1, \dots, n \}$
the Busemann function $b_{\gamma_j}$ is the $j$-th coordinate of $F$.
Let $\xi_j := \gamma_j(\infty)$.
As seen in Lemma \ref{lem: sphstrinf},
for every $x \in X$ the $n$-tuple
$\left( (\xi_1)_x', \dots, (\xi_n)_x' \right)$ of the directions in $\Sigma_xX$
forms a $\delta$-spherical $n$-tuple.
By Lemma \ref{lem: bsm},
it suffices to prove that $F$ is injective.

Suppose that for some distinct $y, z \in X$ we have $F(y) = F(z)$.
Let $\gamma \colon [a,b] \to X$ be a geodesic from $y$ to $z$.
By Proposition \ref{prop: 1dstr},
we can choose $t \in [a,b)$ sufficiently close to $b$
such that the point $z$ is a $(1,\delta)$-strainer at $\gamma(t)$;
in particular, 
the direction $z_{\gamma(t)}' \in \Sigma_{\gamma(t)}X$
is $\delta$-spherical in $\Sigma_{\gamma(t)}X$.
From Proposition \ref{prop: diffbsm} it follows that
$\left\vert \left( F \circ \gamma \right)_+'(t) \right\vert < 6\delta \sqrt{n}$;
more precisely, for all $j \in \{ 1, \dots, n \}$ we have
$\left\vert \cos \angle_{\gamma(t)} (\xi_j,z) \right\vert < 6\delta$.
By Lemma \ref{lem: multi-sph},
the $(n+1)$-tuple 
$\left( (\xi_1)_{\gamma(t)}', \dots, (\xi_n)_{\gamma(t)}', z_{\gamma(t)}' \right)$
of $\delta$-spherical directions in $\Sigma_{\gamma(t)}X$
is $12\delta$-spherical;
in other words,
the point $\gamma(t)$ is $(n+1,12\delta)$-strained.
This property together with Lemma \ref{lem: strdim} implies
$\dim T_{\gamma(t)}X \ge n+1$.
Hence $\dim X \ge n+1$.
This is a contradiction.
\end{proof}

Combining Propositions \ref{prop: submbsm} and \ref{prop: bilipbsm0},
we conclude the following:

\begin{prop}\label{prop: bilipbsm}
For a given constant $N \in \N$,
for every $\epsilon \in (0,1)$, and
for every $n \in \N$,
there exists $\delta \in (0,\infty)$ satisfying the following property: 
Let $X$ be an $N$-doubling,
proper, geodesically complete $\CAT(0)$ space of $\dim X = n$,
and let $F \colon X \to \R^n$ be a Busemann $(n,\delta)$-strainer map.
Then $F$ is a $(1+\epsilon)$-bi-Lipschitz homeomorphism.
\end{prop}

\subsection{Asymptotic regularity}

We denote by $\Delta_1^{n-1}$ the standard spherical $(n-1)$-simplex 
in $\Sph^{n-1}$ defined by
\[
\Delta_1^{n-1} :=
\left\{ \, (u_1, \dots, u_n) \in \Sph^{n-1} \mid 
u_1 \ge 0, \dots, u_n \ge 0 \, \right\}.
\]
We also denote by $\rad \Delta_1^{n-1}$ the radius of $\Delta_1^{n-1}$
defined by
\[
\rad \Delta_1^{n-1} := \inf_{u \in \Delta_1^{n-1}} \sup_{v \in \Delta_1^{n-1}} 
d_{\Sph^{n-1}}(u,v),
\]
where $d_{\Sph^{n-1}}$ is the metric on $\Sph^{n-1}$.
We notice that $\rad \Delta_1^{n-1} < \pi/2$.

We now prove the following asymptotic geometric regularity:

\begin{thm}\label{thm: ageomreg0}
For every $n \in \N$,
there exists $\delta \in (0,\infty)$ depending only on $n$
satisfying the following property:
Let $X$ be a proper, geodesically complete $\CAT(0)$ space.
If 
$\partial_{\T}X$ satisfies 
$d_{\GH} \left( \partial_{\T}X, \Sph^{n-1} \right) < \delta$,
then $X$ is $2\sqrt{n}$-bi-Lipschitz homeomorphic to $\R^n$.
\end{thm}

\begin{proof}
Let $\delta \in (0,\infty)$ be small enough.
By the assumption,
there exists a $2\delta$-approximation
$\varphi \colon \Sph^{n-1} \to \partial_{\T}X$.
Let $(e_1, \dots, e_n)$ be an orthonormal $n$-tuple of points in $\Sph^{n-1}$,
and let $(-e_1, \dots, -e_n)$ be the antipodal one.
For $j \in \{ 1, \dots, n \}$,
we put $\xi_j := \varphi(e_j)$ and $\eta_j := \varphi(-e_j)$.
Then the $n$-tuples $(\xi_1, \dots, \xi_n)$ and
$(\eta_1, \dots, \eta_n)$ are opposite $(n,10\delta)$-strainers at infinity.
By Lemma \ref{lem: sphstrinf},
for every $x \in X$,
the $m$-tuples
$\left( (\xi_1)_x', \dots, (\xi_n)_x' \right)$ and 
$\left( (\eta_1)_x', \dots, (\eta_n)_x' \right)$
are opposite $10\delta$-spherical points in $\Sigma_xX$.

For each $j \in \{ 1, \dots, n \}$,
we take a ray $\gamma_j$ in $X$ with $\xi_j = \gamma_j(\infty)$.
For the $n$-tuple $(\gamma_1, \dots, \gamma_n)$ of the rays in $X$,
we obtain a Busemann $(n,10\delta)$-strainer map
$F \colon X \to \R^n$
with coordinate $(b_{\gamma_1}, \dots, b_{\gamma_n})$.
By Lemma \ref{lem: bsm},
the map $F$ is $2\sqrt{n}$-Lipschitz and $2\sqrt{n}$-open.

We are going to prove that $F$ is injective.
Suppose that for some distinct $y, z \in X$
we have $F(y) = F(z)$.
Let $\gamma \colon [a,b] \to X$ be a geodesic from $y$ to $z$.
By the same way as that discussed in the proof of 
Proposition \ref{prop: bilipbsm0},
we see that for $t \in [a,b)$ sufficiently close to $b$,
the $(n+1)$-tuple 
$\left( (\xi_1)_{\gamma(t)}', \dots, (\xi_n)_{\gamma(t)}', z_{\gamma(t)}' \right)$
of $10\delta$-spherical directions in $\Sigma_{\gamma(t)}X$
is $120\delta$-spherical;
in other words,
the point $\gamma(t)$ is $(n+1,120\delta)$-strained.
Since $X$ is geodesically complete,
we can find $\zeta \in \partial_{\T}X$ such that
the geodesic $yz$ parametrized by $\gamma$
is contained in the ray $y\zeta$ from $y$ to $\zeta$.
By Lemma \ref{lem: multi-sph} (1),
for all $j \in \{ 1, \dots, n \}$ we have
\[
d_{\T}(\xi_j,\zeta) \ge \angle (\xi_j,\zeta) \ge
\angle_{\gamma(t)} \left( (\xi_j)_{\gamma(t)}', z_{\gamma(t)}' \right)
> \frac{\pi}{2} - 240\delta.
\]
The $m$-tuples
$\left( (\xi_1)_{\gamma(t)}', \dots, (\xi_n)_{\gamma(t)}' \right)$ and 
$\left( (\eta_1)_{\gamma(t)}', \dots, (\eta_n)_{\gamma(t)}' \right)$
are opposite $10\delta$-spherical points in $\Sigma_{\gamma(t)}X$.
By Lemmas \ref{lem: opposite} and \ref{lem: multi-sph} (1), 
\begin{align*}
d_{\T}(\eta_j,\zeta) &\ge \angle (\eta_j,\zeta) \ge
\angle_{\gamma(t)} \left( (\eta_j)_{\gamma(t)}', z_{\gamma(t)}' \right) \\
&\ge \angle_{\gamma(t)} \left( (\xi_j)_{\gamma(t)}', (\eta_j)_{\gamma(t)}' \right)
- \angle_{\gamma(t)} \left( (\xi_j)_{\gamma(t)}', z_{\gamma(t)}' \right)
> \frac{\pi}{2} - 130\delta.
\end{align*}
Since the map $\varphi$ is a $2\delta$-approximation,
we can choose a point $u_0 \in \Sph^{n-1}$ 
with $d_{\T}(\varphi(u_0),\zeta) < 4\delta$.
Then for the metric $d_{\Sph^{n-1}}$ on $\Sph^{n-1}$ we have
\[
d_{\Sph^{n-1}}(e_j,u_0) > \frac{\pi}{2} - 250\delta,
\quad
d_{\Sph^{n-1}}(-e_j,u_0) > \frac{\pi}{2} - 140\delta.
\]
On the other hand,
for every $u \in \Sph^{n-1}$ 
we find $k \in \{ 1, \dots, n \}$ such that
$d_{\Sph^{n-1}}(e_k,u) \le \rad \Delta_1^{n-1}$
or 
$d_{\Sph^{n-1}}(-e_k,u) \le \rad \Delta_1^{n-1}$
for the radius $\rad \Delta_1^{n-1}$ of 
the standard spherical $(n-1)$-simplex $\Delta_1^{n-1}$ in $\Sph^{n-1}$. 
This yields a contradiction,
provided $\delta$ is sufficiently small,
since we have $\rad \Delta_1^{n-1} < \pi/2$.
Therefore we see that $F$ is injective.

This finishes the proof of Theorem \ref{thm: ageomreg0}.
\end{proof}

From Proposition \ref{prop: bilipbsm}
we derive the following:

\begin{thm}\label{thm: ageomreg}
For a given constant $N \in \N$,
for every $\epsilon \in (0,1)$, and
for every $n \in \N$,
there exists $\delta \in (0,\infty)$ 
satisfying the following property:
Let $X$ be an $N$-doubling, proper, geodesically complete $\CAT(0)$ space.
If $\partial_{\T}X$ satisfies 
$d_{\GH} \left( \partial_{\T}X, \Sph^{n-1} \right) < \delta$,
then $X$ is $(1+\epsilon)$-bi-Lipschitz homeomorphic to $\R^n$.
\end{thm} 

\begin{proof}
Let $\delta \in (0,\infty)$ be sufficiently small.
From the assumption on $d_{\GH}$,
as discussed in the proof of Theorem \ref{thm: ageomreg0},
we can find a Busemann $(n,10\delta)$-strainer map $F \colon X \to \R^n$
that is a $2\sqrt{n}$-bi-Lipschitz homeomorphism.
In particular, 
we have $\dim X = n$.
Proposition \ref{prop: bilipbsm} leads to the conclusion.
\end{proof}


\section{Asymptotic topological regularity}

In this section,
we prove Theorems \ref{thm: 1d}, \ref{thm: 3/2}, \ref{thm: just3/2}, 
and \ref{thm: 3/2d}.

\subsection{Proof of Theorem \ref{thm: 1d}}

Let $\epsilon \in (0,\infty)$ and $n \in \N$ be arbitrary.
Let $\delta \in (0,1)$ be sufficiently small.
Let $X$ be a purely $n$-dimensional, proper, geodesically complete $\CAT(0)$ space.
Assume that we have $\G_0^n(X) < 1 + \delta$.
By Lemma \ref{lem: chaccat0},
for some constant $N_0 \in \N$ depending only on $n$,
the space $X$ is $N_0$-doubling.
From Proposition \ref{prop: chaccat}
it follows that $X$ has the Gromov--Hausdorff tangent cone $C_{\infty}X$;
moreover, 
\[
\frac{\Haus^{n-1} \left( \partial_{\T} X \right)}{\Haus^{n-1} \left( \Sph^{n-1} \right)} 
< 1 + \delta.
\]
Proposition \ref{prop: accat} implies that
$\partial_{\T}X$ 
is a purely $(n-1)$-dimensional, compact, geodesically complete $\CAT(1)$ space.
Theorem \ref{thm: volreg} leads to that
for some function $\vartheta_n \colon [0,\infty) \to (0,\infty)$
depending only on $n$ with $\lim_{t \to 0} \vartheta_n(t) = 0$,
the Tits boundary
$\partial_{\T}X$ is 
$(1+\vartheta_n(\delta))$-bi-Lipschitz homeomorphic to $\Sph^{n-1}$,
provided $\delta$ is small enough.
Therefore
we have
$d_{\GH} \left( \partial_{\T}X, \Sph^{n-1} \right) < \vartheta_n(\delta)$.
From Theorem \ref{thm: ageomreg}
we deduce that $X$ is $(1+\epsilon)$-bi-Lipschitz homeomorphic to $\R^n$.
This completes the proof.
\qed

\subsection{Manifold recognitions and sphere theorems}

We quote the local topological regularity theorem 
for $\CAT(\kappa)$ spaces
established by Lytchak and the author \cite[Theorem 1.1]{lytchak-nagano2}
in the following form:

\begin{thm}\label{thm: loctopreg}
\emph{(\cite{lytchak-nagano2})}
Let $W$ be an open subset of a proper $\CAT(\kappa)$ space $X$.
Then the following are equivalent:
\begin{enumerate}
\item
$W$ is a topological $n$-manifold;
\item
for every $x \in W$
the space $\Sigma_xX$ is homotopy equivalent to $\Sph^{n-1}$;
\item
for every $x \in W$
the space $T_xX$ is homeomorphic to $\R^n$.
\end{enumerate}
\end{thm}

We say that a triple of points in a $\CAT(1)$ space is a
\emph{tripod}
if the three points have pairwise distance at least $\pi$.
Lytchak and the author \cite{lytchak-nagano2} proved a 
\emph{capacity sphere theorem}
for $\CAT(1)$ spaces
stating that
if a compact, geodesically complete $\CAT(1)$ space
admits no tripod,
then it is homeomorphic to a sphere
(\cite[Theorem 1.5]{lytchak-nagano2}).
As its application,
\cite{lytchak-nagano2} showed
the following volume sphere theorem
for $\CAT(1)$ spaces
(\cite[Theorem 8.3]{lytchak-nagano2},
and \cite{nagano3} for the case of $m = 2$]):

\begin{thm}\label{thm: volsph}
\emph{(\cite{lytchak-nagano2})}
If a purely $m$-dimensional, compact, geodesically complete
$\CAT(1)$ space $\Sigma$ satisfies 
\[
\Haus^m \left( \Sigma \right) < \frac{3}{2} \Haus^m \left( \Sph^m \right),
\]
then $\Sigma$ is homeomorphic to $\Sph^m$.
\end{thm}

The assumption of $\Haus^m$ in Theorem \ref{thm: volsph} is optimal
since the spherical join $\Sph^{m-1} \ast T$ satisfies 
$\Haus^m \left( \Sph^{m-1} \ast T \right) 
= (3/2) \Haus^m \left( \Sph^m \right)$.
We can construct a $\CAT(1)$ $m$-sphere admitting a tripod
whose $m$-dimensional Hausdorff measure is equal 
to $(3/2) \Haus^m \left( \Sph^m \right)$.

\begin{exmp}\label{exmp: triplex}
(\cite{nagano4})
The spherical join $\Sph^{m-2} \ast T$
can be represented by the quotient metric space
$\bigsqcup_{j=1,2,3} \Sph_{+,\, j}^{m-1} / \sim$
obtained by gluing three closed unit $(m-1)$-hemispheres 
$\Sph_{+, \, j}^{m-1}$
along their boundaries 
$\partial \Sph_{+, \, j}^{m-1} = \partial \Sph_{+, \, k}^{m-1}$.
For $j = 1, 2, 3, 3+1=1$, 
let $\Sigma_j^{m-1}$ be
the isometrically embedded unit $(m-1)$-spheres 
$\Sph_{+, \, j}^{m-1} \sqcup \Sph_{+, \, j+1}^{m-1} / \sim$ 
in $\Sph^{m-2} \ast T$
obtained by the relation
$\partial \Sph_{+, \, j}^{m-1} = \partial \Sph_{+, \, j+1}^{m-1}$.
We take three copies of closed unit $m$-hemispheres 
$\Sph_{+, \, j}^m$ for $j \in \{ 1, 2, 3 \}$.
Let $\Sigma$ be the quotient metric space obtained as
\[
\Sigma := \left( \Sph^{m-2} \ast T \right) 
\sqcup \left( \bigsqcup_{j=1,2,3} \Sph_{+, \, j}^m \right) / \sim
\]
by attaching $\Sph_{+, \, j}^m$ to $\Sph^{m-2} \ast T$
along $\Sigma_j^{m-1} = \partial \Sph_{+, \, j}^m$ for each $j \in \{ 1, 2, 3 \}$.
We call $\Sigma$ the 
\emph{$m$-triplex}.
The $m$-triplex $\Sigma$ is a purely $m$-dimensional, compact, 
geodesically complete
$\CAT(1)$ space that is homeomorphic to $\Sph^m$.
This space has a tripod and satisfies
$\Haus^m \left( \Sigma \right) = (3/2) \Haus^m \left( \Sph^m \right)$.
We notice that the $1$-triplex is by definition a circle of length $3\pi$.
\end{exmp}

The author obtained the following characterization
(\cite[Theorem 1.1]{nagano4},
and \cite{nagano3} for the case of $m = 2$) of $\CAT(1)$ spaces of small volume:

\begin{thm}\label{thm: critical}
\emph{(\cite{nagano4})}
If a purely $m$-dimensional, compact, 
geodesically complete $\CAT(1)$ space $\Sigma$ satisfies 
\[
\Haus^m \left( \Sigma \right) = \frac{3}{2} \Haus^m \left( \Sph^m \right),
\]
then $X$ is either
homeomorphic to $\Sph^m$ or
isometric to $\Sph^{m-1} \ast T$.
If in addition $X$ has a tripod,
then $X$ is isometric to either the $m$-triplex
or $\Sph^{m-1} \ast T$.
\end{thm}

\subsection{Proof of Theorem \ref{thm: 3/2}}

Take a purely $n$-dimensional, proper, 
geodesically complete $\CAT(0)$ space $X$.
Assume that $\G_0^n(X) < 3/2$.
By Proposition \ref{prop: chaccat},
the space $X$ is doubling,
and it has the Gromov--Hausdorff asymptotic cone $C_{\infty}X$;
moreover,
for every $x \in X$ we have
\[
\frac{\Haus^{n-1} \left( \Sigma_xX \right)}
{\Haus^{n-1} \left( \Sph^{n-1} \right)}
\le \frac{\Haus^{n-1} \left( \partial_{\T}X \right)}
{\Haus^{n-1} \left( \Sph^{n-1} \right)}
= \G_0^n(X) < \frac{3}{2}.
\]
By Proposition \ref{prop: pure},
the space $\Sigma_xX$ is purely $(n-1)$-dimensional.
The volume sphere theorem \ref{thm: volsph}
implies that $\Sigma_xX$ is homeomorphic to $\Sph^{n-1}$.
Due to the local topological regularity theorem \ref{thm: loctopreg},
we conclude that $X$ is a topological $n$-manifold.

We may assume $n \ge 3$.
From Proposition \ref{prop: accat}
we deduce that $\partial_{\T}X$ 
is a purely $(n-1)$-dimensional, compact, 
geodesically complete $\CAT(1)$ space.
The volume sphere theorem \ref{thm: volsph}
implies that $\partial_{\T}X$ is homeomorphic to $\Sph^{n-1}$.
From Theorem \ref{thm: sci}
we conclude that $X$ is homeomorphic to $\R^n$.
This finishes the proof of Theorem \ref{thm: 3/2}.
\qed

\subsection{Asymptotic volume regularity}

Lytchak \cite[Corollary 1.4]{lytchak3} showed that
if there exists a surjective $1$-Lipschitz map
from a compact spherical building of dimension $m$
onto a geodesically complete $\CAT(1)$ space $Y$,
then $Y$ is a spherical building of dimension $m$.
Notice that every spherical building is 
a geodesically complete $\CAT(1)$ space.
We refer the readers to \cite[Chapter II.10 Appendix]{bridson-haefliger}
for basics on spherical (and Euclidean) buildings.

We show the following volume regularity of $\CAT(1)$ spaces,
which is a generalization of \cite[Proposition 7.1]{nagano2}
(see also Theorem \ref{thm: volreg}):

\begin{prop}\label{prop: volregcat1}
For some function $\vartheta_m \colon [0,\infty) \to (0,\infty)$
depending only on $m$ with $\lim_{t \to 0} \vartheta_m(t) = 0$,
the following holds:
Let $Y$ be a purely $m$-dimensional, proper, 
geodesically complete $\CAT(1)$ space.
Let $Z$ be a compact spherical building of dimension $m$.
Let $f \colon Y \to Z$ be a surjective $1$-Lipschitz map.
If for $\epsilon \in (0,\infty)$ we have
\begin{equation}
\Haus^m \left( Y \right) < \Haus^m \left( Z \right) + \epsilon,
\label{eqn: volregcat1a}
\end{equation}
then $f$ is a $\vartheta_m(\epsilon)$-approximation.
In particular,
if $\Haus^m \left( Y \right) \le \Haus^m \left( Z \right)$,
then $Y$ is isometric to $Z$.
\end{prop}

\begin{proof}
Suppose that for $\epsilon \in (0,\infty)$ we have
\eqref{eqn: volregcat1a}.
Let $y_1, y_2 \in Y$ satisfy 
$d_Z \left( f(y_1), f(y_2) \right) < d_Y \left( y_1, y_2 \right)$,
where $d_Y$ and $d_Z$ are the metrics on $Y$ and on $Z$,
respectively.
Set
\[
s_0 := \frac{d_Z \left( f(y_1), f(y_2) \right)}{2},
\quad
t_0 := \frac{d_Y \left( y_1, y_2 \right)}{2}.
\]
It suffices to show $t_0 - s_0 < \vartheta_m(\epsilon)$
for some $\vartheta_m(\epsilon)$.

Since $Z$ is a spherical building of dimension $m$,
there exists a closed $\pi$-convex subset $\Sigma$ of $Z$
containing $f(y_1), f(y_2)$ such that
$\Sigma$ is isometric to $\Sph^m$.
Take a geodesic $f(y_1)f(y_2)$ in $\Sigma$,
and the midpoint $z \in f(y_1)f(y_2)$ between $f(y_1)$ and $f(y_2)$.
For each $j \in \{ 1, 2 \}$,
we put 
\[
U_j := U_{t_0}(y_j) \cap f^{-1}(\Sigma),
\quad
\bar{U}_j := U_{t_0}(f(y_j)) \cap \Sigma.
\]
From Proposition \ref{prop: pure}
it follows that $\Sigma_{y_j}Y$ is purely $(m-1)$-dimensional.
We can take a surjective $1$-Lipschitz map
$\varphi_{y_j} \colon \Sigma_{y_j}Y \to \Sigma_{f(y_j)}\Sigma$
onto $\Sigma_{f(y_j)}\Sigma$ that is isometric to $\Sph^{m-1}$
(\cite[Proposition 11.3]{lytchak-nagano1}).
We define a map $\Phi_{y_j} \colon U_j \to \bar{U}_j$ by
$\Phi_{y_j}(y) := \exp_{f(y_j)} d_Y(y_j,y) \varphi_{y_j}(y_{y_j}')$,
where $\exp_{f(y_j)}$ is the exponential map from $T_{f(y_j)}\Sigma$.
Then $\Phi_{y_j}$ is surjective.
Since $Y$ is $\CAT(1)$,
the map $\Phi_{y_j}$ is $1$-Lipschitz.
Hence we have
\begin{equation}
\Haus^m \left( \bar{U}_j \right) \le \Haus^m \left( U_j \right).
\label{eqn: volregcat11}
\end{equation}
From the choices of $U_j$ and $\bar{U}_j$
it follows that
$\Sigma - \left( \bar{U}_1 \cup \bar{U}_2 \right)$
is contained in
$f \left( f^{-1}(\Sigma) - (U_1 \cup U_2) \right)$.
Since $f$ is $1$-Lipschitz,
we have
\begin{equation}
\Haus^m \left( \Sigma - \left( \bar{U}_1 \cup \bar{U}_2 \right) \right)
\le \Haus^m \left( f^{-1}(\Sigma) - (U_1 \cup U_2) \right).
\label{eqn: volregcat12}
\end{equation}

Now we put $\bar{U} := U_{t_0-s_0}(z) \cap \Sigma$.
Note that $\bar{U}$ is contained in $\bar{U}_1 \cap \bar{U}_2$.
Since $U_1 \cap U_2$ is empty,
from \eqref{eqn: volregcat11} and \eqref{eqn: volregcat12} we derive
\begin{align*}
\Haus^m \left( \Sigma \right) &\le
\Haus^m \left( \Sigma - \left( \bar{U}_1 \cup \bar{U}_2 \right) \right) +
\Haus^m \left( \bar{U}_1 \right) + \Haus^m \left( \bar{U}_2 \right)
- \Haus^m \left( \bar{U} \right) \\
&\le \Haus^m \left( f^{-1}(\Sigma) - (U_1 \cup U_2) \right)
+ \Haus^m \left( U_1 \right) + \Haus^m \left( U_2 \right) 
- \Haus^m \left( \bar{U} \right) \\
&= \Haus^m \left( f^{-1}(\Sigma) \right) - \Haus^m \left( \bar{U} \right),
\end{align*}
and hence we obtain
\begin{align*}
\Haus^m \left( Z \right) &= \Haus^m \left( \Sigma \right) 
+ \Haus^m \left( Z-\Sigma \right) \\
&\le \Haus^m \left( f^{-1}(\Sigma) \right) - \Haus^m \left( \bar{U} \right)
+ \Haus^m \left( f^{-1}(Z-\Sigma) \right) \\
&= \Haus^m \left( Y \right) - \Haus^m \left( \bar{U} \right).
\end{align*}
Thus by \eqref{eqn: volregcat1a}
we have $\Haus^m \left( \bar{U} \right) < \epsilon$,
so
$\Haus^m \left( U_{t_0-s_0}(z) \cap \Sigma \right) < \epsilon$.
This implies that
$t_0 - s_0 < \vartheta_m(\epsilon)$
holds for some $\vartheta_m(\epsilon)$.
\end{proof}

For $k \in \N$ with $k \ge 2$,
we denote by $T_k$ the discrete metric space consisting of $k$ points
with pairwise distance $\pi$.

As an application of Proposition \ref{prop: volregcat1},
we show the following asymptotic volume regularity:

\begin{prop}\label{prop: volregcat0} 
For some function $\vartheta_n \colon [0,\infty) \to (0,\infty)$
depending only on $n$ with $\lim_{t \to 0} \vartheta_n(t) = 0$,
the following holds:
Let $X$ be a purely $n$-dimensional,
doubling, proper, geodesically complete $\CAT(0)$ space.
Assume that for some $p \in X$ 
the space $\Sigma_pX$ of directions at $p$ is isometric to 
$\Sph^{n-2} \ast T_k$ for some $k$.
If for $\epsilon \in (0,\infty)$ we have
\begin{equation}
\Haus^{n-1} \left( \partial_{\T}X \right) 
< \Haus^{n-1} \left( \Sph^{n-2} \ast T_k \right) + \epsilon,
\label{eqn: volregcat0a}
\end{equation}
then we have
$d_{\GH} \left( \partial_{\T}X, \Sph^{n-2} \ast T_k \right) 
< \vartheta_n(\epsilon)$.
If in addition we have
$\Haus^{n-1} \left( \partial_{\T}X \right) 
\le \Haus^{n-1} \left( \Sph^{n-2} \ast T_k \right)$,
then 
$\partial_{\T}X$ is isometric to $\Sph^{n-2} \ast T_k$;
in particular,
$X$ is isometric to $\R^{n-1} \times C_0(T_k)$.
\end{prop}

\begin{proof}
Proposition \ref{prop: accat} implies that
$\partial_{\T}X$ is 
a purely $(n-1)$-dimensional, compact, geodesically complete $\CAT(1)$ space.
Since $\Sigma_pX$ is isometric to $\Sph^{n-2} \ast T_k$,
it is a compact spherical building of dimension $n-1$.
Take the $1$-Lipschitz map $f_p \colon \partial_{\T}X \to \Sigma_pX$ 
defined by
$f_p(\xi) := \xi_p'$.
By the geodesical completeness of $X$,
the map $f_p$ is surjective.
Since we have \eqref{eqn: volregcat0a} for $\epsilon$,
Proposition \ref{prop: volregcat1} implies that
$f_p$ is a $\vartheta_n(\epsilon)$-approximation
for some $\vartheta_n(\epsilon)$,
and hence 
$d_{\GH} \left( \partial_{\T}X, \Sph^{n-2} \ast T_k \right) 
< \vartheta_n(\epsilon)$.

Assume in addition that we have
$\Haus^{n-1} \left( \partial_{\T}X \right) 
\le \Haus^{n-1} \left( \Sph^{n-2} \ast T_k \right)$.
Proposition \ref{prop: volregcat1} implies that
$\partial_{\T}X$ is isometric to $\Sph^{n-2} \ast T_k$.
From Proposition \ref{prop: tits-splitting},
we deduce that
$X$ is isometric to $X_1 \times X_2$
for some proper, geodesically complete $\CAT(0)$ spaces $X_1$ and $X_2$
such that $\partial_{\T}X_1$ is isometric to $\Sph^{n-2}$
and $\partial_{\T}X_2$ consists of $k$ points.
Then $X_1$ is isometric to $\R^{n-1}$.
From the existence of the point $p$ at which
$\Sigma_pX$ is isometric to $\Sph^{n-2} \ast T_k$,
we see that $X_2$ is isometric to $C_0(T_k)$.
Thus $X$ is isometric to $\R^{n-1} \times C_0(T_k)$.
\end{proof}

\subsection{Proof of Theorem \ref{thm: just3/2}}

Let us consider a purely $n$-dimensional, 
proper, geodesically complete $\CAT(0)$ space $X$.
Assume that we have $\G_0^n(X) = 3/2$.
By Proposition \ref{prop: chaccat},
the space $X$ is doubling,
and it has the Gromov--Hausdorff asymptotic cone $C_{\infty}X$;
moreover,
for every $x \in X$ we have
\[
\frac{\Haus^{n-1} \left( \Sigma_xX \right)}
{\Haus^{n-1} \left( \Sph^{n-1} \right)}
\le \frac{\Haus^{n-1} \left( \partial_{\T}X \right)}
{\Haus^{n-1} \left( \Sph^{n-1} \right)}
= \G_0^n(X) = \frac{3}{2}.
\]
From Proposition \ref{prop: pure} it follows that
$\Sigma_xX$ is purely $(n-1)$-dimensional.
Theorems \ref{thm: volsph} and \ref{thm: critical} imply that
$\Sigma_xX$ is either homeomorphic to $\Sph^n$ 
or isometric to $\Sph^{n-2} \ast T$;
if in addition $\Sigma_xX$ has a tripod,
then $\Sigma_xX$ is isometric to either the $n$-triplex 
or $\Sph^{n-2} \ast T$.
From Proposition \ref{prop: accat}
we deduce that $\partial_{\T}X$ 
is a purely $(n-1)$-dimensional, compact, 
geodesically complete $\CAT(1)$ space.
Theorem \ref{thm: critical} implies that
$\partial_{\T}X$ is either homeomorphic to $\Sph^n$ or 
isometric to $\Sph^{n-2} \ast T$;
if in addition $\partial_{\T}X$ has a tripod,
then $\partial_{\T}X$ is isometric to either the $n$-triplex 
or $\Sph^{n-2} \ast T$.

Suppose first that
$\partial_{\T}X$ is isometric to $\Sph^{n-2} \ast T$.
Then from Proposition \ref{prop: tits-splitting}
we conclude that
$X$ is isometric to $\R^{n-1} \times C_0(T)$.

Suppose next that $\partial_{\T}X$ is homeomorphic to $\Sph^n$,
and $\partial_{\T}X$ has no tripod.
In this case, for every $x \in X$,
the space $\Sigma_xX$ has no tripod too
because of the existence of a surjective $1$-Lipschitz map
from $\partial_{\T}X$ onto $\Sigma_xX$.
In particular,
we have
$\Haus^{n-1} \left( \Sigma_xX \right) 
< (3/2) \Haus^{n-1} \left( \Sph^{n-1} \right)$.
Due to Theorem \ref{thm: volsph},
we conclude that
$\Sigma_xX$ is homeomorphic to $\Sph^{n-1}$.
From the local topological regularity theorem \ref{thm: loctopreg}
it follows that $X$ is a topological $n$-manifold.
Then $X$ is homeomorphic to $\R^n$;
indeed,
in the case of $n \ge 3$,
from Theorem \ref{thm: sci}
we derive the conclusion.

Suppose next that $\partial_{\T}X$ is homeomorphic to $\Sph^n$,
and $\partial_{\T}X$ has a tripod.
In this case,
$\partial_{\T}X$ is isometric to the $n$-triplex.
Then $X$ admits no point at which the space of directions
is isometric to $\Sph^{n-1} \ast T$;
indeed,
if we would find a point $p \in X$
at which $\Sigma_pX$ is isometric to $\Sph^{n-1} \ast T$,
then Proposition \ref{prop: volregcat0} implies that
$\partial_{\T}X$ would be isometric to $\Sph^{n-1} \ast T$ too.
Hence for every $x \in X$
the space $\Sigma_xX$ is homeomorphic to $\Sph^{n-1}$.
From Theorem \ref{thm: loctopreg}
we see that $X$ is a topological $n$-manifold.
Using Theorem \ref{thm: sci},
we conclude that $X$ is homeomorphic to $\R^n$.

Thus we have completed the proof of Theorem \ref{thm: just3/2}.
\qed

\subsection{CAT$\boldsymbol{(\kappa)}$ homology manifolds}

Let $H_{\ast}$ denote the singular homology with $\Z$-coefficients.
A locally compact, separable metric space $M$ is said to be a
\emph{homology $n$-manifold} 
if for every $p \in M$ 
the local homology
$H_{\ast}(M,M-\{ p \})$ at $p$
is isomorphic to $H_{\ast}(\R^n,\R^n-\{ 0 \})$,
where $0$ is the origin of $\R^n$.
A homology $n$-manifold $M$ is a 
\emph{generalized $n$-manifold} 
if $M$ is an $\ANR$ of $\dim M < \infty$.
Every generalized $n$-manifold has dimension $n$.
Due to the theorem of Moore (see \cite[Chapter IV]{wilder}),
for each $n \in \{ 1, 2 \}$,
every generalized $n$-manifold is a topological $n$-manifold.

Every $\CAT(\kappa)$ homology $n$-manifold 
is a geodesically complete generalized $n$-manifold.
We refer the readers to \cite{lytchak-nagano2} for advanced studies
of homology manifolds with an upper curvature bound.

We recall the following stability 
(\cite[Lemma 3.3]{lytchak-nagano2}):

\begin{lem}\label{lem: stabhm}
\emph{(\cite{lytchak-nagano2})}
Assume that a sequence $(X_i,p_i)$ of pointed
proper geodesically complete $\CAT(\kappa)$ spaces
converges to
some proper geodesically complete $\CAT(\kappa)$ space
$(X,p)$ in the pointed Gromov--Hausdorff topology.
If each $X_i$ is a homology $n$-manifold,
then so is $X$.
\end{lem}

We now show the following:

\begin{prop}\label{prop: tits-hm}
Let $X$ be a doubling, proper, geodesically complete $\CAT(0)$ space.
If $X$ is a homology $n$-manifold,
then $\partial_{\T}X$ is a compact homology $(n-1)$-manifold
and $C_0(\partial_{\T}X)$ is a homology $n$-manifold.
\end{prop}

\begin{proof}
From Propositions \ref{prop: accat} and \ref{prop: chaccat} it follows that
for every $p \in X$
the space
$X$ has the Gromov--Hausdorff asymptotic cone 
$\left( C_{\infty}X, p_{\infty} \right)$ isometric 
to $\left( C_0(\partial_{\T}X), 0 \right)$,
where $p_{\infty}$ is the limit base point of $p$.
Assume that $X$ is a homology $n$-manifold.
By Lemma \ref{lem: stabhm},
the cone $C_0(\partial_{\T}X)$ is a homology $n$-manifold.
Since $C_0(\partial_{\T}X) - \{0\}$ is homeomorphic 
to $\partial_{\T}X \times \R$,
we see that $\partial_{T}X$ is a homology $(n-1)$-manifold.
\end{proof}

For $\CAT(1)$ homology manifolds,
the author proved the following volume sphere theorem
(\cite[Theorem 1.2]{nagano4}):

\begin{thm}\label{thm: volsphhm}
\emph{(\cite{nagano4})}
For every $m \in \N$,
there exists $\delta \in (0,\infty)$ depending only on $m$
such that
if a compact $\CAT(1)$ homology $m$-manifold $\Sigma$ satisfies
\[
\Haus^m \left( \Sigma \right) 
< \frac{3}{2} \Haus^m \left( \Sph^m \right) + \delta,
\]
then $\Sigma$ is homeomorphic to $\Sph^m$.
\end{thm}

\subsection{Proof of Theorem \ref{thm: 3/2d}}

Let $\delta \in (0,1)$ be sufficiently small.
Let $X$ be a complete $\CAT(0)$ homology $n$-manifold.
Assume that we have $\G_0^n(X) < 3/2 + \delta$.
By Proposition \ref{prop: chaccat},
the space $X$ is doubling,
and it has the Gromov--Hausdorff asymptotic cone $C_{\infty}X$;
moreover,
for every $x \in X$ we have
\[
\frac{\Haus^{n-1} \left( \Sigma_xX \right)}
{\Haus^{n-1} \left( \Sph^{n-1} \right)}
\le \frac{\Haus^{n-1} \left( \partial_{\T}X \right)}
{\Haus^{n-1} \left( \Sph^{n-1} \right)}
= \G_0^n(X) < \frac{3}{2} + \delta.
\]
By Proposition \ref{prop: pure},
the space $\Sigma_xX$ is purely $(n-1)$-dimensional.
The volume sphere theorem \ref{thm: volsphhm} for homology manifolds
implies that $\Sigma_xX$ is homeomorphic to $\Sph^{n-1}$.
From the local topological regularity theorem \ref{thm: loctopreg},
we see that $X$ is a topological $n$-manifold.

We may assume $n \ge 3$.
From Proposition \ref{prop: tits-hm}
we deduce that $\partial_{\T}X$ 
is a compact $\CAT(1)$ homology $(n-1)$-manifold.
The volume sphere theorem \ref{thm: volsphhm} for homology manifolds
implies that $\partial_{\T}X$ is homeomorphic to $\Sph^{n-1}$.
From Theorem \ref{thm: sci}
we conclude that $X$ is homeomorphic to $\R^n$.
This finishes the proof of Theorem \ref{thm: 3/2d}.
\qed




\end{document}